\documentclass[12pt]{article}
\usepackage{amssymb,amsmath,amscd,epsf}
\usepackage{graphicx}
\usepackage[dvipsnames]{xcolor}
\usepackage[bookmarks=false]{hyperref}
\hypersetup{
    colorlinks=true,
    linkcolor=Sepia,
    citecolor=Mahogany,
    urlcolor=Blue, 
}

\usepackage[active]{srcltx} 

\usepackage{xcolor}

\def\re{\color{black}}

\renewcommand{\sharp}{\natural}

\newcounter{ar}

\newcounter{ab}

\newcommand{\mysection}[1]{\section{\large\bf #1}}

\newtheorem{defn}[subsection]{Definition}
\newtheorem{prop}[subsection]{Proposition}
\newtheorem{theo}[subsection]{Theorem}
\newtheorem{lemma}[subsection]{Lemma}

\newtheorem{corll}[subsection]{Corollary}

\newtheorem{lolemma}[subsection]{Local Lemma}

\newenvironment{rem}{\smallskip\noindent%
\refstepcounter{subsection}%
{\bf \thesubsection}~~{\sc Remark.}\hspace{-1mm}}
{\smallskip}

\newenvironment{ex}{\smallskip\noindent%
\refstepcounter{subsection}%
{\bf \thesubsection}~~{\sc Example.}}
{\smallskip}

\newenvironment{proof}[1]{\noindent {\em Proof#1.}}
{~$\square$\smallskip}

\newenvironment{ackn}{\medskip \noindent \small
{\sl Acknowledgments.}}{\bigskip}

\newenvironment{smallbibl}[1]
{\small
\begin{thebibliography}{#1}}
{\end{thebibliography}}

\newcommand{\dbcoh}{{\mathcal D}^b_{\rm coh}}
\newcommand{\Hom}{{\rm Hom}}

\renewcommand{\mod}{{\rm mod}}
\renewcommand{\d}{{\rm d}}
\renewcommand{\P}{{\mathcal P}}

\newcommand{\id}{{\rm id}}
\newcommand{\Cvect}{{\rm C}^{\cdot}({\rm Vect})}

\newcommand{\ZZ}{{\mathbb Z}}
\newcommand{\CC}{{\mathbb C}}
\newcommand{\RR}{{\mathbb R}}
\newcommand{\QQ}{{\mathbb Q}}

\newcommand{\HH}{{\mathbb H}}

\renewcommand{\O}{{\mathcal O}}
\newcommand{\C}{{\mathcal C}}
\newcommand{\A}{{\mathcal A}}
\newcommand{\E}{{\mathcal E}}
\newcommand{\F}{{\mathcal F}}
\newcommand{\G}{{\mathcal G}}
\newcommand{\D}{{\mathcal D}}
\newcommand{\M}{{\mathcal M}}
\newcommand{\U}{{\mathcal U}}
\newcommand{\Ao}{{\mathcal A}_{\omega}}
\newcommand{\Fo}{{\mathcal F}_{\omega}}
\newcommand{\Eo}{{\mathcal E}_{\omega}}

\newcommand{\Ho}{{\mathcal H}o}
\newcommand{\Rhom}{{\mathbb R}{\rm Hom}}
\newcommand{\RHom}{{\mathbb R}{\mathcal H}om}

\newcommand{\toiso}{\xrightarrow{\sim\;}}

\newcommand{\dbar}{\bar\partial}
\newcommand{\bu}[1][.5]{{\raisebox{0.30ex}{\scalebox{0.5}{\text{$\bullet$}}}}}

\newcommand{\tr}{\mathrm{tr}}
\newcommand{\ch}{\mathrm{ch}}

\newcommand{\barb}{\bar\beta}
\newcommand{\bard}{\bar D}

\def\beq{\begin{equation}}
\def\eeq{\end{equation}}

\begin{document}

\title{\vspace{1.3cm}
\Large\sc Coherent Sheaves, Chern Classes, and Superconnections on Compact Complex-Analytic Manifolds}

\author{
Alexey Bondal
and
Alexei Rosly
}

\date{June 1, 2022}

\maketitle
\begin{center}
    {\it Dedicated to the memory of Igor Shafarevich \\ 
    on the occasion of 
    his 
    100\textsuperscript{\,th} Anniversary}
\end{center}
\begin{abstract}

We construct a twist-closed enhancement of the category $\dbcoh (X)$, the bounded derived category of complexes 
of $\O_X$-modules with coherent cohomology, by means of the DG-category of $\dbar$-superconnections. 
Then we apply the techniques of
$\dbar$-superconnections to define Chern classes and Bott-Chern classes of objects 
in the category, in particular, of coherent sheaves.
\end{abstract}

\newpage

\mysection{\sc Introduction}
The derived category of coherent sheaves is known to be a meaningful homological
invariant of an algebraic variety. The basic motivation for the authors of the
present paper was the wish to understand to which extent the derived category is a
good invariant for complex analytic manifolds.

Let $X$ be a smooth compact complex-analytic manifold and $\dbcoh (X)$ the derived category
of $\O_X$-modules with bounded coherent cohomology. There are some indications that
this category is not as good as it is in the algebraic case. 
First, a result in \cite{Verb1} implies that $\dbcoh (X)$ are equivalent for all K3 surfaces
$X$ with no curves. Hence, the derived category does not `feel' the complex geometry
of the generic K3 surface. 
This situation is very different for algebraic varieties. 
If the canonical or anticanonial class of $X$ is ample, then the variety can be uniquely reconstructed from its derived category \cite{BO}. The number of abelian varieties derived equivalent to a given one is finite \cite{O}.
There
is at most countably many algebraic varieties in a given class of derived equivalence
(\cite{AT}). An example of a projective variety with infinite number of derived partners is a 3-dimensional projective space with 8 suitably chosen points blown up \cite{L}.

Second, a wonderful property of the derived category of coherent sheaves on a smooth
proper algebraic variety is that it satisfies a property similar to Brown
representability. Namely, the category is {\em saturated}, i.e all bounded
cohomological functors with values in vector spaces are representable (see \cite{BK1},
\cite{BVdB}). It was shown in \cite {BVdB} that for the derived category of a smooth
compact complex surface with no curves (say, a generic K3) this property does not
hold. It was also conjectured in {\it loc.cit.} that if the derived category is saturated then the
variety is (an analytification of) an algebraic space. The conjecture was
proven by B.\,To\"en and M.\,Vaqui\'e \cite{TV2}, though they used an {\it a priori} stronger
(DG) version of saturatedness.

It was probably these facts that lead to the common opinion that the
derived category was not a meaningful homological invariant in the complex-analytic
case. Literally, this was correct. However, there was some ambiguity in what
sort of natural structure is reasonable to fix when considering derived categories.
In particular, it was mentioned quite a time ago that it made sense to consider
triangulated categories together with {\em enhancements}, a sort of enrichments of
the categories with a DG-structures \cite{BK2}. Results of this paper suggest that some good
(that is, twist-closed) enhancements are very relevant to complex-analytic geometry of manifolds.

We construct a DG-enhancement, $\C_X$, of $\dbcoh (X)$ by $\dbar$-superconnections.
Its objects are DG-modules  over the DG-algebra of Dolbeault forms, with suitable properties.
Our first result is that the homotopy category of this DG-category is equivalent to
$\dbcoh (X)$. A similar enhancement was independently considered by J. Block \cite{Block}. In particular, we discuss the fully faithfulness of the functor that gives the equivalence. Also we do not assume that every object in 
$\dbcoh (X)$ has a presentation by a finite complex of locally free sheaves
of $\O_X$-modules (as in the proof of Lemma 4.6 in {\it loc.cit}). A counter-example of C.\ Voisin \cite{Voisin}
shows that this does not hold for all complex-analytic manifolds. The work by N.\ Pali \cite{P1}, \cite{P2} where he described coherent sheaves in terms of $\dbar$-connections on special sheaves of modules over smooth functions was inspiring for us at the early stage of this project.

An important property of the category $\C_X$ is that it is {\em twist-closed}, which
means that the functors constructed via twisted complexes in the category are
representable. Twisted complexes are solutions of the Maurer-Cartan equation with
values in the (degree one) endomorphisms of objects in the DG-category. Note that
twist-closed categories are pre-triangulated, {\it i.e.} so-called one-sided twisted
complexes are representable, but it is crucial for constructing moduli of objects in the 
complex-analytic set-up to have representability of all twisted complexes.

Using the dg-enhancement via superconnections we discuss Chern character for objects in $\dbcoh (X)$.
To this end, we extend any $\dbar$-superconnection to a (non-flat) De Rham superconnection 
in the sense of Quillen \cite{Q} via a suitable choice of Hermitian metrics. 
As a result, we see that the $p$-th coefficient of the Chern character for an object in $\dbcoh (X)$ 
is represented by a closed  $(p,p)$-form. 

This suggests the idea to define characteristic classes of objects in $\dbcoh (X)$ which lie in Bott-Chern cohomology. 
The developed techniques for Chern classes of superconnections allows us to do this in a relatively easy way. 

We leave for another publication the extension of the theory of $\dbar$-superconnections to 
non-compact varieties $X$ and its application to constructing moduli spaces of objects in $\dbcoh (X)$. 

When proving the theorem on the DG-enhancement via $\dbar$-superconnections we follow our Kavli IPMU preprint in 2011 \cite{BR}. 
The applications to Chern and Bott-Chern characteristic classes was presented by the second-named author 
in his talk on the Russian-Japanese conference "Categorical and analytic invariants in Algebraic Geometry II" 
at Kavli IPMU, Tokyo University, in 2015 \cite{Ros}. 
We are pleased to dedicate this work to the memory of a famous mathematician and strong personality which was Igor Rostislavovich Shafarevich. 

Having prepared this paper for publication, we learned about the sizeable preprint 
by J.-M. Bismut, Shu Shen and Zhaoting Wei \cite{BSW}, where the authors addressed similar questions, but also discuss Grothendieck-Riemann-Roch theorem. It would be interesting to compare the approaches of the two papers. 

We thank Mikhail Kapranov and Maxim Kontsevich for valuable comments on the subject of this paper. 

\bigskip
\begin{ackn}
The reported study was funded by RFBR and CNRS, project number 21-51-15005.
	The work of A. I. Bondal was performed at the Steklov International Mathematical Center and supported by the Ministry of Science and Higher Education of the Russian Federation (agreement no. 075-15-2022-265).
	This research was partially supported by World Premier International Research Center Initiative (WPI Initiative), MEXT, Japan.
	This work was supported by JSPS KAKENHI Grant Number JP20H01794.
	The work of A. Rosly has been also funded 
	within the framework of the HSE University Basic Research Program.
\end{ackn}

\mysection{DG-Enhancements}

In this section we recall some facts on DG-categories and introduce the notion of
twist-closed DG-categories.

DG-categories can be considered as a particular class of $A_{\infty}$-categories. The
theory has a direct generalization to the $A_{\infty}$-case. We avoid this more
general context here, because natural enhancements which one meets in complex geometry
have a DG-structure.

Let $\C$ be an additive DG-category over a field. This means that we have finite direct sums
of objects, morphisms between any two objects constitute a $\ZZ$-graded complex of
vector spaces over the field, and Leibniz rule for the composition of morphisms is
satisfied. We will denote the space of degree $i$ morphisms in $\C$ by $\Rhom ^i$.

Two objects $A,B\in \C$ are said to be DG-isomorphic if there exists an invertible
degree zero closed morphism $f\in \Rhom^0 (A, B)$. Accordingly, one defines
DG-isomorphism of functors.

We assume that the category is equipped with an equivalence $T\!:\C\to \C$, called
translation functor, together with a DG-isomorphism of functors:
$$
\mu:\id \to T,
$$
where $\mu$ has degree $-1$.

If $\C$ is a DG-category, then its {\em homotopy category} $\Ho\,\C$ is defined as a
category with the same objects as in $\C$, but morphisms are the degree zero homology
of complexes of morphisms in $\C$.

Now we explain the ideology (based on \cite{BK2}) of twisted complexes and functors, 
which they represent.

A {\it twisted complex} in a DG-category $\C$ is a pair $W=(E, \alpha)$, where $E$ is an
object in $\C$ and $\alpha$, a {\em twisting cochain},  is in $\Rhom ^1(E, E)$ and
satisfies Maurer-Cartan equation:
$$
\d \alpha +\alpha ^2=0.
$$

Any twisted complex $T=(E, \alpha)$ defines a contravariant functor
$$
h^W:\C\to \Cvect,
$$
where $\Cvect$ stands for the category of complexes of vector spaces. The functor is
defined on $A\in \C$ by:
$$
A\mapsto \Rhom^{\bu} (A, E),
$$
but the differential in $\Rhom^{\bu} (A, E)$ is twisted by $\alpha$:
$$
\d_{\alpha}=\d_{\Rhom^{\bu} (A, E)} +\alpha.
$$
The Maurer-Cartan equation for $\alpha$ implies that $\d _{\alpha}^2=0$.

We will also consider twisted complexes of a particular form. A twisted complex
$W=(E, \alpha)$ is called {\em one-sided}, if the object $E$ is decomposed into a
finite direct sum $E=\oplus E_i$ and $\alpha$ is given by a strictly upper triangular matrix with
respect to this decomposition.

The basic example of a one-sided twisted complex is obtained by taking $E$ to be of
the form
\begin{equation}\label{onesided}
E=E_1\oplus E_2
\end{equation}
and $\alpha$
to belong to  $\Rhom ^1(E_1, E_2)$, a subspace in
$\Rhom^1(E, E)$. Note that $\alpha^2$ is automatically zero for such $\alpha$.
Therefore, the Maurer-Cartan equation reduces to $\d \alpha =0$. Thus $\alpha$ can be
interpreted as a closed degree zero morphism $E_1\to E_2[1]$.

In the following definition, DG-categories are assumed to be
additive and equipped with translation functors.
\begin{defn}
\begin{itemize}
\item[(i)] A DG-category is called {\em twist-closed} if $h^W$ is
representable for any twisted complex $W$.
\item[(ii)] A
DG-category is called {\em pre-triangulated} if $h^W$ is
representable for any one-sided twisted complex $W$.
\end{itemize}
\end{defn}

The following theorem was proved in \cite{BK2}.

\begin{theo}
If $\C$ is pre-triangulated, then $\Ho \C$ is naturally
triangulated.
\end{theo}
The idea behind this theorem is that the twisting cochain $\alpha$ in a one-sided
twisted complex of the form (\ref{onesided}) defines a morphism
${\tilde \alpha}:E_1\to E_2[1]$ in $\Ho \C$ and the object that represents the
functor $h^W$ gives a cone of ${\tilde \alpha}$ in the homotopy category. Thus the
basic hereditary problem of the axiomatics of triangulated categories that the cones
of morphisms are not canonical is resolved by lifting morphisms in a triangulated
category to closed morphisms in an appropriate DG-category.

Another reason why the DG-context looks to be more suitable is that the pre-triangulatedness is
a property of a DG-category, while to make a category triangulated one has to put an
extra structure on the category (to fix a class of exact triangles). The price to pay
for transferring into the DG-world is that one has to consider DG-categories up to an
appropriate equivalence relation, i.e. there is always a variety of equivalent
choices for DG-categories "representing" a given triangulated category.

\begin{defn} If $\D$ is a triangulated category, then a
pre-triangulated category $\C$ together with an equivalence of triangulated
categories $\Ho\,\C{\to} \D$ is said to be an {\em enhancement} of $\D$. The category
$\D$ is then said to be enhanced. A functor between two DG-categories is said to be
a {\em quasi-equivalence} if it induces an equivalence of the corresponding homotopy
categories.
\end{defn}

It is clear from definitions that a twist-closed DG-category is pre-triangulated,
hence its homotopy category is naturally triangulated. It will be crucial for our
further constructions to have enhancements which are twist-closed.

{\bf NB!} The twist-closedness is not preserved under
quasi-equivalences.

The following example shows that a standard enhancement of the
derived category of coherent sheaves on an algebraic variety is
not twist-closed.

\begin{ex} Let $X$ be an algebraic variety with the structure
sheaf $\O_X$ and $\D =\dbcoh (X)$ the derived category of complexes of $\O_X$-modules
with bounded coherent cohomology. Consider the DG-category $\C={\rm I}(X)$ of bounded
below complexes of injective $\O_X$-modules with bounded coherent cohomology. By
definition, this is a full subcategory in the DG-category $C^{\bu}(\O_X\mbox{-mod})$ of
complexes of $\O_X$-modules. It is known that ${\rm I}(X)$ (not the
$C^{\bu}(\O_X\mbox{-mod})$) is an enhancement of $\dbcoh (X)$ \cite{BK2}, \cite{BLL}.

Let $E$ be a complex in ${\rm I}(X)$ with the differential $\d$, such that some terms
of $E$ are not coherent $\O_X$-modules (typically, neither of them is). Consider the
twisted complex $W=(E,\alpha)$ with $\alpha =-\d$. One can easily see that the functor
$h^W$ is not representable. Indeed, it is represented by the object in
$C^{\bu}(\O_X\mbox{-mod})$ which is a complex with the same graded components as $E$ and
with trivial differential. Its cohomology is not coherent.

\end{ex}

Here is an example of a twist-closed enhancement.

\begin{ex} Let $\D=\D^b (\mbox{mod-$A$})$ be the bounded derived category of
(right) modules over an
associative algebra $A$ of finite global
dimension. Consider the DG-category $\C=\P(A)$ consisting of perfect complexes, i.e.
finite complexes of finitely generated projective $A$-modules. This is a full
subcategory in the DG-category $C^{\bu}(\mbox{mod-$A$})$ of all complexes of $A$-modules.
Again, $\P(A)$ (and not $C^{\bu}(\mbox{mod-$A$})$) is an enhancement of $\D^b (\mbox{mod-$A$})$.
Every twisted complex $W=(E,\alpha)$ over this category produces a functor $h^W$
which is representable by the same graded module $E$ but with the new differential
$$
d_T=d_E+\alpha .
$$
This is a perfect complex. Hence the enhancement is twist-closed.
\end{ex}

\mysection{An Enhancement via $\dbar$-superconnections}

We are going to construct a twist-closed enhancement of $\dbcoh (X)$.
It will be a DG-category $\C=\C_X$ whose homotopy category is
equivalent to $\dbcoh (X)$. The idea of the construction can be explained via Koszul
duality applied to the algebra of differential operators ({\it cf.}\ \cite{Kap}). This goes
along the following lines.

\subsection{\sc A Viewpoint via Koszul Duality}

Denote $\A ^{i,j}=\A ^{i,j}_X$ the sheaf of smooth
complex-valued $(i,j)$-forms on $X$. For the sake of simplicity,
we will also use the notation $\A ^i=\A^i_X =\A^{0,i}$. In
particular, $\A^0_X$ is the sheaf of smooth functions on $X$. 
We regard $\A^{\bu}$ as a (sheaf of) dg-rings endowed with Dolbeault's differential $\dbar$, and 
use the notation $\A^\sharp$ for the same ring considered as a graded ring with no differential. 
We use
notation $\A^+=\A_X^+=\oplus_{i>0} \A ^{i}_X$ for the positive part of Dolbeault
complex. If
$\E$ is a locally free $\A^0$-module, then
$\A^{i,j}(\E)=\A_X^{i,j}(\E)=\A_X^{i,j}\otimes _{\A^0}\E$ denotes the sheaf of smooth
$(i,j)$-forms on $X$ with values in $\E$. Similarly, we put $\A (\E)=\A_X
(\E)=\A_X\otimes\E$.

A locally free coherent sheaf $E$ on $X$ is given by a smooth vector bundle $\E $ on $X$
with a flat $\dbar$-connection $\bar\nabla$, the sheaf $E$ being the sheaf of $\bar\nabla$-horizontal sections. 
One can interpret such a connection as a
module over the sheaf of algebras ${\mathcal D}^{\dbar}_X$ of $\dbar$-differential
operators on $X$. Algebra ${\mathcal D}^{\dbar}_X$ is filtered by degree of differential operators. 
The associated graded algebra is the symmetric algebra $S^{\bu}(T_X^{0,1})$ over $\A^0_X$ 
of the sheaf $T_X^{0,1}$ of vector fields of type $(0,1)$ on $X$. 
The quadratic dual to it over $\A^0_X$ is the  sheaf of graded algebras of smooth $(0,i)$-forms on $X$:
$$
\A_X^{\sharp}= \oplus \A ^{i}_X.
$$
A version of the non-homogeneous quadratic Koszul duality (\textit{cf.} \cite{Kap}) implies that the 
Koszul dual to ${\mathcal D}^{\dbar}_X$ itself is 
$$
\A^{\bu}_X=(\A_X^{\sharp}, \dbar ), 
$$
i.e. $\A^{\bu}_X$ regarded as a sheaf of DG-algebras 
equipped with Dolbeault differential $\dbar$. 

Note that $\A^0_X$ has a natural structure of a ${\mathcal D}^{\dbar}_X$-module. 
Besides, $\A^0_X$ is the zero component of the filtration
on the algebra ${\mathcal D}^{\dbar}_X$\,.  
Then, the Koszul duality at the level of algebras amounts to the following statement.

\begin{lemma}\label{Koszula}
If $X$ is a complex-analytic manifold of dimension $n$, then there is quasiisomorphism of sheaves of DG-algebras:
$$
\RHom_{{\mathcal D}^{\dbar}_X} (\A^0_X, \A^0_X)\simeq \A^{\bu}_X.
$$
\end{lemma}

\textit{Proof.}~ Consider the ${\dbar}$-version of the Spencer complex (see \cite{Sabbah}),
{\it i.e.} a locally free resolution of $\A^0_X$  by left  ${\mathcal D}^{\dbar}_X$-modules:
$$
0\to {\mathcal D}^{\dbar}_X\otimes_{\A^0_X} \Lambda^nT_X^{0,1}\to\ldots 
\to {\mathcal D}^{\dbar}_X\otimes_{\A^0_X} \Lambda^2T_X^{0,1} \to {\mathcal D}^{\dbar}_X\otimes_{\A^0_X} T_X^{0,1} 
\to {\mathcal D}^{\dbar}_X \to \A^0_X\to 0.
$$
Applying this resolution to calculating the required $\RHom_{{\mathcal D}^{\dbar}_X} (\A^0_X, \A^0_X)$ gives the result. $\square$

A result of N.~Pali  (when slightly reformulated) identifies the abelian category of coherent $\O_X$-modules 
with a suitable subcategory of ${\mathcal D}^{\dbar}_X$-modules.

\begin{theo}\label{Palitheo} {\em\cite{P1}}
For a complex-analytic manifold $X$, the category of coherent $\O_X$-modules 
is equivalent to the category of ${\mathcal D}^{\dbar}_X$-modules which, as $\A^0_X$-modules, 
locally have a finite resolution by finite rank free $\A^0_X$-modules.
\end{theo}

Koszul duality typically states that appropriate derived categories of modules over
Koszul dual algebras are equivalent. Thus, Lemma \ref{Koszula} together with Theorem \ref{Palitheo} give a reason to search for an enhancement of
$\dbcoh (X)$ among DG-categories of suitable $\A_X$-DG-modules. Note that Pali's proof uses a choice of norms and Leray-Koppelman operators to prove the existence of solutions of a complicated system of differential equations. The results of the present paper imply Pali's theorem with no use of such involved analytic techniques. 

{\re For the case of a locally free $\O_X$-module $E$, 
take $\E=\A^0\otimes_{\O_X}E$\,. Then, 
the $\dbar$-connection $\bar\nabla :\E\to \A^{0,1}(\E )$ can be extended to a differential in
$\A(\E )$. As a result, $\A(\E )$ acquires a natural structure of a DG-module over
$\A $. This is the module which corresponds to the sheaf $E$ of $\bar\nabla$-horizontal sections, }
suggested by Koszul duality.

We extend this correspondence to arbitrary objects in $\dbcoh (X)$. It is
natural to call $\A$-DG-modules of suitable type by $\dbar$-superconnections (see
below). 
It is easy to construct a
superconnection corresponding to an arbitrary object in $\dbcoh (X)$ 
if every coherent sheaf on $X$ has a resolution by locally free
$\O_X$-modules, which is true if $X$ is the analyticification 
of an algebraic variety
(the resolution can be constructed via an {\it ample system} of line bundles \cite{Il})
or if $X$ is a compact analytic surface as 
proved by Schuster \cite{Schuster}.

In general, locally free resolutions do not exist. A counterexample with a complex
torus is due to C. Voisin \cite{Voisin}. This makes the issue more subtle.

\subsection{\sc The category of $\dbar$-superconnections}

Before we start, let us agree on some notation concerning graded modules over 
supercommutative graded rings. 
Let $\mathcal{R}$ be such a ring, and $\M$ and $\mathcal{N}$ two graded left modules over $\mathcal{R}$. 
We denote by $\Hom_{\mathcal{R}}(\M,\mathcal{N})$ a left graded module of homomorphisms from $\M$ to $\mathcal{N}$ 
which satisfy $\forall\phi\in\Hom_{\mathcal{R}}(\M,\mathcal{N})$, $\forall\mu\in\M$, $\forall\omega\in\mathcal{R}$,
$$
  \phi(\omega\cdot\mu)=(-1)^{\deg\phi\cdot\deg\omega}\omega\cdot\phi(\mu)\,.
$$
Note that every left graded $\mathcal{R}$-module can be made a right module by defining
$$
    \mu\cdot\omega:=(-1)^{\deg\mu\cdot\deg\omega}\omega\cdot\mu \,.
$$
Then, what we defined as $\Hom_{\mathcal{R}}(\M,\mathcal{N})$ becomes simply homomorphisms which are 
$\mathcal{R}$-linear with respect to the right-multiplication.

Let $M$ be a bounded (left) DG-module over $\A$ (more precisely, a sheaf of modules). 
If $\bard$ is a differential in $M$, $s$ a section of $M$, and $\omega$ an element in $\A$,
then the Leibniz rule has to be satisfied:
\begin{equation}\label{leibniz}
\bard(\omega\cdot s)=\dbar\omega \cdot s+(-1)^{\deg
\omega}\omega\cdot \bard s
\end{equation}
Forgetting the differential, but not the grading in $M$, we obtain a graded module, $M^{\sharp}$, over
$\A^{\sharp}$. 
Define the objects of the category $\C=\C_X$
to be DG-modules $M$ over $\A$ for which $M^{\sharp}$ are locally free graded
modules  of finite rank over $\A^{\sharp}$. If $M$ is in $\C$, then it is, in
particular, a bounded graded complex of locally free $\A^{0}$-modules. We say that
objects of $\C$ are flat\footnote{
Until Section \ref{Chern}, we deal mainly with flat superconnections. Therefore, we will 
frequently omit the word 'flat'.}
$\dbar$-{\em superconnections}, which is inspired by Quillen's terminology in \cite{Q}.

Let $M$ and $N$ be
in $\C$. Then the graded sheaf
${\mathcal Hom}_{\A^{\sharp}}(M^{\sharp},N^{\sharp})$ is endowed with the structure 
of a left $\A^{\sharp}$-module and a differential  
$\bard_{\mathcal{H}om}(\phi)=\bard_N\circ\phi-(-1)^{\deg\phi}\phi\circ\bard_M$\,, 
which are related by the Leibniz rule. 
That is to say, sections
$\phi\in {\mathcal Hom}^{\bu}_{\A}(M,N)$, $\omega\in \A$, and $s\in M$ satisfy the
sign rule:
\begin{eqnarray}\label{gradhom}
    \phi (\omega\cdot s)=(-1)^{\deg\phi\cdot\deg\omega}\omega\cdot\phi(s),\\ 
    \bard_{\mathcal{H}om}(\omega\cdot\phi)=\dbar\omega\cdot\phi+
    (-1)^{\deg\omega}\omega\cdot\bard_{\mathcal{H}om}(\phi)\,.
\end{eqnarray}
Thus, the complex
${\mathcal Hom}^{\bu}_{\A}(M, N)=({\mathcal Hom}_{\A^{\sharp}}(M^{\sharp},N^{\sharp}),\bard_{\mathcal{H}om})$ 
is also an object in $\C$.
We define morphisms in the dg-category $\C$ as a complex of global sections of
${\mathcal Hom}_\A^{\bu}(M, N)$:
$$
\RR\Hom^{\bu}_{\C}(M, N):=\Gamma (X, {\mathcal Hom}^{\bu}_{\A}(M,
N))\,,
$$
equipped with the differential $\bard_{\Rhom}$ inherited from $\bard_{\mathcal{H}om}$\;.
\begin{prop}
$\C$ is a twist-closed DG-category.
\end{prop}

\noindent
\textit{Proof.}~
The translation functor is defined by the shift of grading of
DG-$\A$-modules. Hence, we need to prove that every twisted complex in $\C$ is
representable.

Let $M$ be in $\C$, $\bard$ the differential in $M$, and $\alpha\in \RR\Hom ^1_{\C }(M,M)$ a
solution of the Maurer-Cartan equation, $\bard_{\mathcal{H}om}\,\alpha+\alpha^2=0$. 
Since $\alpha$ satisfies the sign rule \eqref{gradhom} with $\deg\alpha=1$,
then $\bard+\alpha$ satisfies the Leibniz rule (\ref{leibniz}) and $(\bard+\alpha)^2=0$.
Hence $M(\alpha)$, the same $\A^{\sharp}$-module $M^{\sharp}$, but with a new
differential, $\bard+\alpha$, is again a DG-$\A$-module. Clearly, it belongs to $\C$ and is 
a representing object for the twisted complex defined by the pair $(M, \alpha )$.
$\square$

\begin{corll}
Category $\C$ is pre-triangulated. The homotopy category $\Ho (\C )$ is triangulated.
\end{corll}

Any object in $\C_X$ is by restriction along the embedding $\O_X \to \A ^0_X$ a complex of $\O_X$-modules, 
because $\O_X$ is $\dbar$-horizontal. If $M$ and $N$ are in
$\C_X$, then a closed morphism in $\RR \Hom^0_{\C} (M, N)$ obviously defines a morphism of
complexes of the corresponding $\O_X$-modules. Homotopy-equivalent closed morphisms
define isomorphic morphisms in $\D^b(\O_X\mbox{-mod})$, because the derived category factors
through the homotopy category of complexes of $\O_X$-modules. Hence we obtain a
functor:
\begin{equation}\label{functor}
\Phi :\Ho (\C_X) \to\D^b(\O_X\mbox{-mod}).
\end{equation}

The rest of this section is devoted to proving that $\C_X$ defines
an enhancement of $\dbcoh(X)$ via the functor $\Phi$.

\begin{theo} \label{equiv}
Let $X$ be a compact smooth complex-analytic manifold. Then $\Phi$ is a triangulated
equivalence between the homotopy category $\Ho (\C_X)$ and $\dbcoh (X)$.
\end{theo}

\noindent
The strategy of the proof is as follows. First, 
we will show that any $\dbar$-super\-connection is a complex of $\O_X$-modules with
bounded coherent cohomology. 
Second, 
we will show that functor $\Phi$ is
fully faithful, i.e it retains homomorphisms between any two objects in $\Ho (\C_X )$.
Third, we will prove that any object in $\dbcoh (X)$ is
quasi-isomorphic to a $\dbar$-superconnection. The proof relies on the description of the
local structure of $\dbar$-superconnections.

\subsection{\sc The local structure of $\dbar$-superconnections}

Every finite complex of locally free $\O_X$-modules $E^{\bu}$ yields a
$\dbar$-superconnection by taking a tensor product of complexes
$$
\A_X\otimes_{\O_X}E^{\bu}\,
$$
with $\dbar\otimes\id_E$ as a differential. 
Here we will prove the technical statement which claims that any flat 
$\dbar$-superconnection is locally isomorphic (gauge-equivalent) 
to a $\dbar$-superconnection of this kind. We
believe that this fact is of independent interest.

Let $M$ be a $\dbar$-superconnection. Since $M^{\sharp}$ is locally free over $\A^{\sharp}$,
it can be non-canonically presented in the form:
\begin{equation}\label{decompmodule}
M^{\sharp}=\A^{\sharp}\otimes_{\A_X^0}\E^{\bu}
\end{equation}
where $\E^{\bu}=M/\A^+M$  
is a finitely generated graded locally free $\A_X^0$-module 
($\A^+\!\!=\oplus_{i>0} \A ^{i}_X$). To have
this presentation, choose an $\A_X^0$-linear splitting of the projection $M\to \E^{\bu}$ and use
the $\A^{\sharp}$-module structure of $M^{\sharp}$. Every $\E^i$ can be understood as
a smooth complex vector bundle on $X$. Consider the (non-canonical) bigrading of
$M^{\sharp}$ where $\A^i\otimes \E^j$ has bidegree $(i, j)$. The initial (canonical)
grading of
$\A^i\otimes \E^j$ in $M$ has the total degree $(i+j)$.

The differential $\bard$ in this module has the following decomposition with respect to
this bigrading\footnote{
The (bi)degree of a homogeneous operator acting in a (bi)graded space is that how much it changes the (bi)degree 
of any element it acts upon.}
\begin{equation}\label{expansion}
\bar D=\bar\gamma+\bar\nabla+\sum_{i\geqslant 2}\bar \beta_i\,.
\end{equation}
Here $\bar\gamma$ is the component of bidegree $(0,1)$, 
$\bar\nabla$ the component of
bidegree $(1,0)$, and $\bar\beta_i$ the component of degree $(i,1-i)$. The Leibniz rule
implies that $\bar\gamma$ is an endomorphism of $M^{\sharp}$ of total degree 1, hence it satisfies
the sign rule (\ref{gradhom}).
Therefore, $\bar\gamma$ is fully defined by $\A_X^0$-module homomorphisms
$\bar\gamma_j: \E^j\to \E^{j+1}$. 
$\bar\nabla$ can be understood as a set of not necessarily flat
$\dbar$-connections $\bar\nabla_j$ on $\E^j$. An important point is that these
connections are not necessarily flat (one can say, these are only homotopically flat, 
\textit{cf.}, the 3-d line of eq.\ \eqref{D2=0} below), 
which does not allow us at this stage to
endow $\E^j$ with the structure of a holomorphic vector bundle. By the Leibniz rule again, the components
$\bar\beta_i$'s 
correspond to $\A_X^0$-morphisms
$   \E^j\to \A^i\otimes\E^{j-i+1}$. In this
notation, the (bi)graded components of the condition
\begin{equation}
    \bar D^2=0
\end{equation}
read as a sequence of equations:
\begin{equation}\label{D2=0}
\begin{split}
    \bar\gamma ^2 & = 0\,, \\
    [\bar\gamma,\bar\nabla &] = 0\,, \\
    \bar\nabla^2+[\bar\gamma,&\bar\beta _2 ]  = 0\,, \\
    [\bar\nabla , \bar\beta_2 ]+&[\bar\gamma , \bar\beta _3]  =  0\,, \\
    [\bar\nabla , \bar\beta_3] +\bar\beta_2^{\;\;2}&+[\bar\gamma , \bar\beta_4]  = 0\,,
\end{split}
\end{equation}
and so on. Note that here and from now on, the bracket $[\,\cdot\,,\,\cdot\,]$ denotes a 
supercommutator, so, in eq.\ \eqref{D2=0}, it is actually an anticommutator, because 
all the participants are of odd degree in this case, for example,
$[\bar\gamma,\bar\nabla]= \bar\gamma\bar\nabla + \bar\nabla\bar\gamma$.

Note that if all $\bar\beta_i$'s were zero  these equations would be equivalent to
saying that $\bar\nabla$  give a holomorphic structure (that is a flat $\bar\partial$-connection) 
in all vector bundles $\E^j$'s, and, then, $\bar\gamma$
would be a holomorphic differential in a complex of holomorphic vector bundles.

Now, we choose  a point $x\in X$ and an open neighborhood of $x$ in analytic
topology on $X$. We consider local {\em gauge transformations} of the form:
\begin{equation}\label{gauge}
 \bar D'=e^{-\phi}\bar De^{\phi}
\end{equation}
with $\phi$ an $\A^{\sharp}$-module endomorphism of $M^{\sharp}$, which has degree 0
with respect to the canonical grading. Clearly, $\phi$ is defined by its values on
$\E^{\bu}$. Thus, we interpret $\phi$ as an element of $\Hom^0_{\A^0} (\E^{\bu},
\A^{\sharp}\otimes\E^{\bu})$. 
We say that the {\em gauge parameter} $\phi$ is {\em strict} if it has a decomposition
\begin{equation}\label{phi1+}
\phi=\phi_1+\phi_2+\dots
\end{equation}
where 
$\phi_i\in\oplus_j\Hom^0_{\A^0}(\E^{j},\A^{i}\otimes\E^{j-i})$ 
over the neighborhood of $x$. 
In other words, for a strict gauge parameter, the component 
$\phi_0\in \Hom^0_{\A^0}(\E^{\bu}, \E^{\bu}) $ is zero. Note, that this condition does not depend on the non-canonical splitting \eqref{decompmodule}.

The corresponding gauge 
transformation, $\exp\phi$, will be also called {\em strict}. Such gauge transformations can be regarded as a change 
of the non-canonical bigrading of $M$, which is the same as a choice of an isomorphism $\A^{\sharp}\otimes\E^{\bu} \toiso M^{\sharp}$, which in turn is the same as an automorphism $\A^{\sharp}\otimes\E^{\bu} \toiso\A^{\sharp}\otimes\E^{\bu}$ 
trivial on $\E^0 =\A^0\otimes\E^0$.

Transformation (\ref{gauge}) for components of $\bar D'$  reads:
$$
\bar\gamma '=\bar\gamma,
$$
i.e. $\gamma$ does not change,
$$
\bar\nabla '=\bar\nabla + [\bar\gamma , \phi_1],
$$
\begin{equation}\label{trlaw}
    \bar\beta_2'=\bar\beta_2+[\bar\gamma , \phi_2]+\frac12\bar\gamma\phi_1^2+
\frac12\phi_1^2\bar\gamma-\phi_1\bar\gamma \phi_1+[\bar\nabla , \phi_1 ],
~~etc.
\end{equation}

The following lemma confirms that every superconnection is locally isomorphic to a
complex of holomorphic vector bundles 
(\textit{cf.}, \cite[Lemma 4.5]{Block}).

\begin{lolemma} \label{local} 
Any flat $\dbar$-superconnection over a polydisc can be transformed by a strict gauge
transformation of the form (\ref{gauge}) to the form (\ref{expansion}) with all
${\bar \beta_i}$'s being zero. 
\end{lolemma}

The proof will easily follow from a technical lemma \ref{tlemma} below. An immediate consequence of the above Local Lemma is the following

\begin{corll}\label{loc-cohomology} 
    (i) Every  $\dbar$-superconnection on a polydisc is isomorphic to $\A_X\otimes_{\O_X}E^{\bu}$, where $E^{\bu}$ is a finite complex of free $\O_X$-modules, \\ 
    (ii) Cohomology sheaves of any $\dbar$-superconnection are coherent sheaves of $\O_X$-modules.
\end{corll}
\textit{Proof.} By Lemma \ref{local}, for any $\dbar$-superconnection $M$ we can find a strict gauge transformation that gives an operator $\bar{D}$ of the form (\ref{expansion}) with all ${\bar \beta_i}$'s being zero. This implies that $M$ is isomorphic to $\A_X\otimes_{\O_X}E^{\bu}\,$, where $\bar\gamma$ plays the role of the differential in $E^{\bu}$ and $\bar\nabla$ consists of  the $\dbar$-connections
on all $\A_X\otimes_{\O_X}E^i$'s.

The claim \textit{(ii)} is local, thus we can assume $M=\A_X\otimes_{\O_X}E^{\bu}\,$ as in \textit{(i)}. As a complex of $\O_X$-modules
it is quasiisomorphic to $E^{\bu}$, hence it has coherent cohomology.
$\square$

\begin{subsection}{\sc Relative superconnections on the product of polydiscs}
Let $U$, $W$ be complex manifolds and $\pi\!:U\to W$ a projection with smooth fibres. 
Denote by $\A_\pi^{\bu}=\oplus \A_\pi^i$ the sheaf of rings on $U$ of relative $(0,i)$-forms. 
$\A_\pi^{\bu}$ is a dg-algebra with the differential 
$\dbar_\pi$\,, which is given by the relative Dolbeault operator, $\dbar_\pi\!:\A_\pi^i\to\A_\pi^{i+1}$.

Let now $M^{\bu}$ be a dg-module over $\A_\pi^{\bu}$. 
We call $M^{\bu}$ a {\it relative} $\dbar$-{\it superconnection} if it is locally free over $\A_\pi^\sharp$.

Given $U, W, \pi$ be as above, consider a graded smooth vector bundle $\E^{\bu}$ on $U$ endowed 
with a relative $\dbar$-connection $\bar\nabla\!:\E^i\to\A_\pi^1\otimes\E^i$\,, where 
$\bar\nabla$ obeys the Leibniz $\dbar_\pi$-rule.
$\bar\nabla$ can be extended to $\A_\pi^\sharp\otimes\E^{\bu}$ by the graded Leibniz rule 
with respect to $\A_\pi^{\bu}$\,. If $\bar\nabla$ is flat,{\it i.e.} $\bar\nabla^2=0$, we get a relative 
 $\dbar$-superconnection of the form $(\A_\pi^\sharp\otimes\E^{\bu}\,,\,\bar\nabla)$. 
Let us call a superconnection of this particular shape an {\it ordinary} relative 
$\dbar$-connection.

\medskip
In Lemma \ref{tlemma} below, we shall need the following construction. 
If $U=V\times W$ and $\pi_V$, $\pi_W$ are two projections on the factors, 
then we have obvious isomorphisms of relative and absolute forms: 
$\A_{\pi_V}^{\bu}=\pi_W^*\A_W$, $\A_{\pi_W}^{\bu}=\pi_V^*\A_V$\,, and an isomorphism 
$\A_U^{\bu}=\A_V^{\bu}\boxtimes_{\A_U^0}\A_W^{\bu}$\,. 
This agrees with differentials, because the Dolbeault operator on $U$, 
$\dbar_U$\,, can be split into the sum $\dbar_U=\dbar_V+\dbar_W$ 
in an obvious sense. ($\dbar_V$ acts ``along $V$'', that is increases by 1 only 
the grading in $\A_{\pi_W}^{\bu}$; similarly, $\dbar_W$ acts ``along $W$''.)

Let $\E^{\bu}$ be a graded smooth vector bundle over $U$. Suppose we have two 
relative $\dbar$-superconnections of the form 
$(\A_{\pi_W}^\sharp\otimes\E^{\bu}\,,\,\bard_V)$ and 
$(\A_{\pi_V}^\sharp\otimes\E^{\bu}\,,\,\bard_W)$, where $\bard_V$ and $\bard_W$ obey 
the Leibniz rules with respect to the corresponding projections. 
Given this, we can extend $\bard_V$ from $\A_{\pi_W}^{\sharp}\otimes\E^{\bu}$ to 
$\A_U^{\sharp}\otimes\E^{\bu}$ by the Leibniz $\dbar_V$-rule, and similarly for $\bard_W$\,, and construct the  operator $\bard=\bard_V+\bard_W$ on $\A_U^\sharp\otimes\E^{\bu}$
which satisfies the Leibnitz $\dbar_U$-rule.  
If additionally $[\bard_V\,,\,\bard_W]=0$, then we get a $\dbar$-superconnection 
$(\A_U^\sharp\otimes\E^{\bu}\,,\,\bard)$ on $U$.

Now consider an $n$-dimensional polydisc $U$
which is presented as the product $U=V\times W$, where $V$ has 
coordinates $z^1,\ldots,z^m$, $m>1$, and $W$ coordinates $z^{m+1},\ldots,z^n$. 
Let $V'$ be the polydisc with $m-1$ coordinates, $z^1,\ldots,z^{m-1}$.
Thus, $V$ is a product $V=V'\times V_1$\,, where $V_1$ is 1-dimensional polydisc with coordinate $z^m$ as a coordinate. 
Let $W'$ be the polydisc $W'=V_1\times W$.
We have a new decomposition into the product:
$U=V'\times W'$.

\begin{lemma}\label{tlemma}
Let $M$ be a $\dbar$-superconnection over the polydisc $U=V\times W$ as above, $\dim V > 1$.
Suppose the differential in $M$ is of the form
$$
    \bar D=\bar D_V+\bar\nabla_W\,,
$$
where $\bar\nabla_W$ is an ordinary relative $\dbar$-connection along $W$, 
and $\bard_V$ a relative $\dbar$-superconnection along $V$. 
Then there exists a strict gauge transformation $\bar D'=e^{-\phi}\bar De^{\phi}$
such that 
$$
    \bard'=\bard_{V'}+\bar\nabla_{W'} \,,
$$
where 
$\bar\nabla_{W'}$ is an ordinary relative  $\dbar$-connection along $W'$, and  
$\bard_{V'}$ a relative  $\dbar$-superconnection along $V'$.
\end{lemma}

\noindent
\textit{Proof.}~
Since $\bar D=\bar D_V+\bar\nabla_W$\,, it can be written (\textit{cf.}, eq.\,\eqref{expansion}) as 
$$
    \bar D=\bar \gamma+\bar \nabla+\bar\nabla_W+\sum_{i=2}^m\bar \beta_i\,,
$$
where $\bar\gamma\,,~\bar\nabla\,,~\bar\beta_i$'s are operators of degree 1 on sections of 
$M^\sharp$ defined by morphisms 
$$
    \bar\gamma\!:\E^{j}\to\E^{j+1} \,,
$$
$$
    \barb_i\!:\E^{j}\to\pi_V^*\A_V^i\otimes\E^{j-i+1},
$$
and $\bar\nabla$ is a first order differential operator mapping $\E^{j}$ to 
$\pi_V^*\A_V^1\otimes\E^j$\,.

All of these, $\bar\gamma\,,~\bar\nabla\,,~\bar\beta_i$'s, (anti)commute with $\bar\nabla_W$\,, because 
$[\bard_V\,,\bar\nabla_W]=0$.
We shall say, they depend holomorphically on $W$.

Let us now consider $V$ as a product $V=V'\times V_1$\,. 
Any $(0,i)$-form on $V$ 
can be decomposed as
$$
\omega=\omega_{i,0}+\omega_{i-1,1}
$$
where 
$$
\omega_{i-1,1}=\eta _{i-1,0}\d\bar z^m,
$$
and $\omega_{i,0}$ and $\eta_{i-1,0}$ are Dolbeaut forms on $V'$ of degree $i$ and $i-1$ respectively.

Analogously, we have the decomposition of 
$\bar D_V=\bar\gamma+\bar\nabla+\sum_{i=2}^m\bar\beta_i$\,. 
In particular,
$$
    \bar\nabla=\bar\nabla_{1,0} + \bar\nabla_{0,1}\,,
$$
$$
    \barb_k=(\barb_k)_{k,0}+(\barb_k)_{k-1,1} \,.
$$
Note that, since $V_1$ is one-dimensional, any $\dbar$-connection along $V_1$ is 
automatically flat, {\it i.e.} $(\bar\nabla_{0,1})^2=0$. Moreover, the differential $\bar\nabla_{0,1}$ 
has trivial (0,1)-cohomology on a disc (it is essentially the same as the Dolbeault operator).

Let us now prove that there exists a strict gauge transformation $\bar D'_V=e^{-\phi}\bar D_Ve^{\phi}$\,, 
where $\phi=\phi_1+\phi_2+\ldots$\,, such that $\phi_i=(\phi_i)_{i,0}$ ($\phi_i$ contains no differential 
$d\bar z^m$, that is $(\phi_i)_{i-1,1}=0$) and such that, for 
$\bard'_V=\bar\gamma+\bar\nabla'+\sum\barb'_i$\,, we have that 
$(\barb'_i)_{i-1,1}=0$ for $i\geqslant 2$. Moreover, one can require that $\phi_i$'s commute with 
$\bar\nabla_W$ (that is to say, holomorphically depend on $W$).

From the transformation law of $\bard$ (\textit{cf.}, eq.\,\eqref{trlaw}) we get an equation
$$
    (\barb'_2)_{1,1}=(\barb_2)_{1,1}+[\bar\nabla_{0,1}\,,\,(\phi_1)_{1,0}]=0 \,.
$$
As we have mentioned, $\bar\nabla_{0,1}$ is acyclic in positive degree, hence, the above equation 
has a solution $\phi_1=(\phi_1)_{1,0}$\,. Note, that the conditions 
$\phi_0=0$ and $\phi_i=(\phi_i)_{i,0}$ 
guaranty that $\bar\nabla'_{0,1}=\bar\nabla_{0,1}$.

Let us use this result as the base of induction in $k$ and suppose 
$(\barb_s)_{s-1,1}=0$ for $s\leqslant k$. Then, the equation of vanishing of 
$(\barb'_{k+1})_{k,1}$ reads as
\begin{equation}\label{bk}
    (\barb_{k+1})_{k,1}+[\bar\nabla_{0,1}\,,\,(\phi_k)_{k,0}]=0 \,.
\end{equation}
This again has a solution. Moreover, since $\barb_k$'s depend holomorphically on $W$ 
and so is $\bar\nabla_{0,1}$\,, solutions $\phi_k$ to eq.\,\eqref{bk} can also be chosen 
to commute with $\bar\nabla_W$\,.

Thus, we can kill the $(\barb_i)_{i-1,1}$ for all $i=2,\ldots,m$\,. 
At the end, we get 
$$
    \bard'_V=e^{-\phi}\bard_V e^\phi=\bard'_{1,0}+\bar\nabla_{0,1} \,,
$$
where $(\bard'_{1,0})^2=0$, $(\bar\nabla_{0,1})^2=0$, and 
$[\bard'_{1,0}\,,\,\bar\nabla_{0,1}]=0$. 
Since $\phi$ commutes with $\bar\nabla_W$\,, we also have that
$$
    \bard'=e^{-\phi}(\bard_V+\bar\nabla_W) e^\phi=\bard'_{1,0}+\bar\nabla_{0,1}+\bar\nabla_W \,.
$$
Let us rewrite this as 
$$
    \bard'=\bard_{V'}+\bar\nabla_{W'} \,,
$$
where $\bard_{V'}=\bard'_{1,0}$ and $\bar\nabla_{W'}=\bar\nabla_{0,1}+\bar\nabla_W$\,, 
which yields the result.
$\square$

\medskip
\noindent
\textit{Proof of Lemma \ref{local}}~
It follows by use of Lemma \ref{tlemma} and induction in decreasing $m$.
$\square$

\end{subsection}

\subsection{\sc Full Faithfulness}
Let $M$ and $N$ be $\dbar$-superconnections. We regard them as complexes of $\O_X$-modules
and consider the complex of derived local homomorphisms $\RHom_{\O_X} (M, N)$, which
is the object in $D(\O_X\mbox{-mod})$ too.

Fix a complex $I(N)$ of injective $\O_X$-modules together with a
quasi-isomorphism $N\to I(N)$. It induces a morphism of complexes
$$
\mu :{\mathcal Hom}^{\bu}_{\O_X}(M, N)\to {\mathcal
Hom}^{\bu}_{\O_X}(M, I(N)).
$$
Consider the composition $\phi$ of the natural map
$ {\mathcal Hom}^{\bu}_{\A}(M, N)\to {\mathcal Hom}^{\bu}_{\O_X}(M, N)$ with
$\mu$. Since injective sheaves are acyclic with respect to the functor
${\mathcal Hom}(\U , -)$, for every $\O_X$-module $\U$ (cf.\ \cite[Corollary 2.4.2]{KS})
we have a functorial isomorphism in $\D(\O_X\mbox{-mod})$:
\begin{equation}\label{rhom}
\RHom_{\O_X} (M,N)\simeq {\mathcal Hom}^{\bu}_{\O_X}(M, I(N)).
\end{equation}
Hence we have a commutative diagram:

\begin{center}
\parbox{50mm}{
\begin{picture}(200,60)
\put(0,0){${\mathcal Hom}^{\bu}_{\A}(M,N)$}
\put(55,15){\vector(1,1){20}} \put(60,40){${\mathcal
Hom}^{\bu}_{\O_X}(M,N)$} \put(110,35){\vector(1,-1){20}}
\put(120,0){$\RHom_{\O_X}(M,N)$} \put(75,5){\vector(1,0){40}}
\put(90,10){$\phi$}\put(125,26){$\mu$}
\end{picture}}
\end{center}

\begin{lemma} \label{rhomiso}

$\phi$ induces a quasi-isomorphism ${\mathcal Hom}_{\A}^{\bu}(M,
N)\simeq \RHom_{\O_X} (M, N)$.

\end{lemma}

\noindent
\textit{Proof.}~
We need to show that $\phi$ induces an isomorphism of
cohomology sheaves. This is a local statement. Hence, we can assume that $X$ is a polydisc. 
According to Corollary \ref{loc-cohomology} we can replace $M$ and $N$ by Dolbeault
bicomplexes of finite complexes $E_1^{\bu}$ and $E_2^{\bu}$
of free $\O_X$-modules:
$$
M=\A\otimes_{\O_X} E_1^{\bu},\ \ N=\A\otimes_{\O_X} E_2^{\bu}.
$$

Then the complex ${\mathcal Hom}_{\A}^{\bu}(M, N)$ is isomorphic
to Dolbeault complex of the complex of sheaves of local homomorphisms
$\A\otimes_{\O_X} {\mathcal Hom}_{\O_X}(E_1^{\bu}, E_2^{\bu})$.

If $E_1^{\bu}$ and $E_2^{\bu}$ both consist of single
locally free sheaves $E_1$ and $E_2$, then this complex is
obviously quasi-isomorphic to ${\mathcal Hom}_{\O_X}(E_1, E_2)$.
On the other hand, $M$ and $N$ are quasi-isomorphic to, respectively, $E_1$ and
$E_2$, hence $\RHom_{\O_X} (M, N)={\mathcal
Hom}_{\O_X}(E_1, E_2)$, i.e. the statement of the lemma is clear
for this case.

We shall use now the induction on the length of the complexes $E_1^{\bu}$ and
$E_2^{\bu}$. If one of them, say $E_1^{\bu}$, has length greater than one, then
we decompose it into a short exact sequences of complexes
$$
0\to E_1^{\prime \bu}\to E_1^{\bu}\to E_1^{\prime \prime \bu}\to 0
$$
where $E_1^{\prime \bu}$ and $E_1^{\prime \prime \bu}$ are 
complexes of free sheaves of shorter length. Since $\A$ is flat over 
$\O_X$ by Theorem \ref{flat-over-o} in the Appendix, functor $\A\otimes_{\O_X} (-)$ 
is exact on the derived categories. Let 
\begin{equation}\label{MMM}
M^{\prime}\to M\to M^{\prime \prime}
\end{equation}
be the decomposition of the object $M$ induced by the above exact sequence.

By induction, we know that $\phi$ induces quasi-isomorphisms:
$$
{\mathcal Hom}_{\A}^{\bu}(M^{\prime}, N)\simeq \RHom_{\O_X} (M^{\prime}, N),
$$
$$
{\mathcal Hom}_{\A}^{\bu}(M^{\prime \prime}, N)\simeq \RHom_{\O_X} (M^{\prime \prime}, N)
$$
Cohomology sheaves of ${\mathcal Hom}_{\A}^{\bu}(M, N)$ and $\RHom_{\O_X} (M, N)$
fit into two long exact sequences, obtained by applying the functors ${\mathcal
Hom}_{\A}^{\bu}(-, N)$ and $\RHom_{\O_X} (-, N)$ to the triangle (\ref{MMM}), and
$\phi$ gives a morphism of these long sequences. The quasi-isomorphism for $M$
follows from the lemma on 5 homomorphisms applied to this diagram. 
$\square$

\begin{prop}\label{fulfai}
The functor $\Phi :\Ho(\C_X)\to\D^b(\O_X\mbox{-mod})$ is fully faithful.
\end{prop}

\noindent
\textit{Proof.}~
By the standard property of local $\RHom$ applied to two $\dbar$-superconnections $M$ and $N$, 
one can recover the global homomorphisms by the formula:
$$
\Hom_{\D(\O_X\mbox{-}\mathrm{mod})}(M, N)={\HH}^0 (X, \RHom_{\O_X} (M, N)),
$$
where $\HH^i$ stands for the hypercohomology of a complex of $\O_X$-modules. In view
of Lemma \ref{rhomiso}, we can replace the right hand side in this equality and get:
$$
\Hom_{\D(\O_X\mbox{-}\mathrm{mod})}(M,N)={\HH}^0 (X,{\mathcal
Hom}^{\bu}_{\A}(M, N)).
$$
There is a standard spectral sequence converging to the hypercohomology:
$$
{\rm H}^i(X,{\mathcal Hom}^{j}_{\A}(M, N))\to
{\HH}^{i+j}({\mathcal Hom}^{\bu}_{\A}(M, N))
$$

Since sheaves ${\mathcal Hom}^{j}_{\A}(M, N)$ are fine, cohomology
${\rm H} ^i(X,{\mathcal Hom}^{j}_{\A}(M, N))$ are trivial for $i\ge 0$. The spectral
sequence degenerates, and yields the equality of $\Hom_{\D(\O_X\mbox{-}\mathrm{mod})}(M, N)$ 
with the 0-th homology of the complex of global sections 
$\Gamma (X, {\mathcal Hom}^{\bu}_{\A}(M,N))$, which is exactly ${\Hom}_{\Ho \C}(M,N)$.
$\square$

\subsection{\sc Every Coherent Sheaf is Quasiisomorphic 
to \\
a $\dbar$-superconnection}

In this Subsection we construct a $\dbar$-superconnection that is 
quasi-isomorphic to a given coherent sheaf.

Let $\Ao$ be the sheaf of complex-valued real-analytic functions on $X$ 
regarded as a real-analytic manifold.
Since the sheaf $\Ao$ is identified with the restriction to the diagonal
$X\subset X\times{\bar X}$ of the sheaf of holomorphic functions $\O_{X\times\bar X}$\,, 
which is coherent, $\Ao$ is also coherent.

Let $F$ be a coherent sheaf of $\O_X$-modules, \textit{i.e.}\ it is 
finitely presented. Then, $\Fo=\Ao\otimes_{\O_X}F$ is also finitely presented 
($\otimes$ is right-exact). Hence, $\Fo$ is coherent as an $\Ao$-module, because $\Ao$ is coherent. 

\begin{lemma}\label{lem-resol}
For a coherent sheaf of $\O_X$-modules $F$, 
the sheaf of $\A^0$-modules $\F =\A^0\otimes_{\O_X}F$ has a locally free
resolution $\E^{\bu}$ over $\A^0$:
\begin{equation}\label{resol1}
0\to \E^{-n}\to \dots \to \E^{0}\to \F\to 0
\end{equation}
\end{lemma}

\noindent
\textit{Proof.}~
As it is explained by Atiyah and Hirzebruch in \cite{AH}, the famous
Grauert's result \cite{G} on the existence of a fundamental system of Stein neighborhoods of any 
real-analytic manifold $X$ in its complexification, together with theorems A and B of
Cartan, easily implies, on a compact $X$, the existence of a finite locally free resolution for an
arbitrary coherent $\Ao$-module. Fix such a resolution for $\Fo$:

\begin{equation}\label{resolution}
0\to \Eo^{-n}\to \dots \to \Eo^{0}\to \Fo\to 0
\end{equation}

By a result of Malgrange \cite{Malgrange} $\A^0$ is a flat ring over $\Ao$. Hence, 
by taking tensor product with $\A^0$ over $\Ao$\,, we obtain an
$\A^0$-resolution $\E^{\bu}=\A^0\otimes_{\Ao}\Eo^{\bu}$ for $\F$ of the form (\ref{resol1}).

For reader's convenience, we will give another, somewhat more explicit, construction
for the resolution of the form (\ref{resol1}), which, however, still relies on \cite{Malgrange}.
In a sufficiently small neighborhood $U$ of any point $x\in X$,
the sheaf $\F |_U$ has a finite resolution by finitely generated free $\A^0_U$-modules.
Indeed, the sheaf $F|_U$ has a finite resolution by free $\O_U$-modules.
The sheaf $\A^0_U$ is flat over 
$\O_U$ by proposition \ref{flatness} (see Appendix). Thus, we can take the tensor product 
of the resolution with $\A^0_U$ and get the required local
resolution. 

On a compact manifold, 
every $\A^0$-module which is locally generated 
by a finite number of sections is also globally generated by a finite number of sections,
because $\A ^0$ is a fine sheaf. This is clearly applicable to sheaves which locally
have free resolutions. Now if we have an epimorphism between two sheaves which have local
resolutions, then the kernel is a sheaf which also has such a resolution. Apply this
to the epimorphism $(\A^0)^s\to \F$, which exists because $\F$ is generated by a
finite number of global sections. We get that the kernel has also free
resolutions locally on $X$ and is generated over $\A^0$ by finite number of sections. Hence we can
iterate the process and construct the resolution until the kernel becomes locally
free. This must happen after a finite number of iterations, because the manifold,
being compact, is covered by a finite number of open sets on which a finite free
resolution exists. Just take the resolution of the length equal to the maximum of
lengths of these free resolutions on this finite number of open sets.
$\square$

\medskip
Now denote by $\bar\nabla _F$ the differential $\dbar$ acting on $\A\otimes_{\A^0} 
\F=\A\otimes_{\O_X}\!F$ along the
first tensor factor. It makes $\A\otimes_{\O_X}\!F$ into a DG-module over $\A$. It is not a
$\dbar$-superconnection, because in general it is not a locally free as an $\A^\sharp$-module,
but we will construct a $\dbar$-superconnection 
quasi-isomorphic to $\A\otimes_{\O_X}\!F$. 

Choose a resolution $\E^{\bu}$ as in (\ref{resol1}) and put $M^{\sharp}=(\A\otimes \E^{\bu})^{\sharp}$ as a
graded module over $\A^\sharp$ with the grading given by the sum of gradings on $\A$ and $\E^{\bu}$. 
Let $\bar\gamma$ be the differential in the resolution $\E^{\bu}$. We denote also 
$\bar\gamma$ its $\A^\sharp$-linear (as of operator of degree 1) extension to $M^{\sharp}$. We denote 
$\bar\gamma_0$ the map $\E^0\to\F$, its $\A^\sharp$-linear extension to the map
$(\A\otimes \E^0)^{\sharp}\to (\A\otimes\F)^{\sharp}$, and the extension to the map
$M^{\sharp}\to (\A\otimes\F)^{\sharp}$ which is zero on other components, 
$(\A\otimes \E^i)^{\sharp}\to (\A\otimes\F)^{\sharp}$ for $i\ne 0$.
\begin{theo}\label{m-as-superc}
$M^{\sharp}$ can be endowed with a structure of a $\dbar$-superconnection such that $\bar\gamma_0$ 
furnishes a quasi-isomorphism $\bar\gamma_0\!:M\to\A\otimes_{\O_X} F$.
\end{theo}

\noindent
\textit{Proof.}~
We claim
that there exists a system of (non-flat, in general) $\dbar$-connections $\bar\nabla_i$ on
$\E^{i}$'s which commute with $\bar\gamma$ and such that
\begin{equation}\label{nabla-nablaf}
\bar\gamma_0 \bar\nabla_{0} ={\bar \nabla}_F \bar\gamma_0.
\end{equation}
First, we construct a $\dbar$-connection on $\E^0$ compatible with $\bar\gamma_0$, $\bar\nabla_F$ in this sense. 

Applying functor $\Gamma (X,\A^1\otimes_{\A^0} (-))$ to the sequence (\ref{resol1}) gives an exact sequence, 
because the sheaf $\A^1$ is flat over $\A^0$ and all the sheaves in the sequence are fine, hence acyclic.
In particular, $\Gamma (X, \A^1\otimes_{\A^0}\E^0)\to \Gamma (X, \A^1\otimes_{\A^0}\F)$ is an epimorphism.

If $\E^0$ is a 
free $\A^0$-module, i.e. a trivial smooth vector bundle, then take a basis $\{s_i\}$
of its sections and define $\bar\nabla (s_i)$ to be any element in $\A^1\otimes \E^0$
such that $\bar\gamma_0\bar\nabla (s_i)=\bar\nabla_F\bar\gamma_0(s_i)$.  
Such an element exists in view of the above epimorphism on global sections. If
$\E^0$ is not a trivial bundle, then consider the direct sum $S=\E^0\oplus \G$ which is a
trivial vector bundle. Take, as above, a connection $\bar\nabla_S : S\to \A^1\otimes S$ on
$S$ compatible with the composite map $S\stackrel{p}\to \E^0\stackrel{\bar\gamma_0}{\to }\F$, 
where $p: S\to \E^0$ is the projection. Then we have:
$$
\bar\gamma_0 p_0\bar\nabla_S = \bar\nabla_F\bar\gamma_0p_0.
$$
The restriction $\bar\nabla_0=p_0\bar\nabla_S |_{\E^0}$ of this connection to $\E^0$ will define a connection on $\E^0$ compatible
with $\bar\gamma_0$, that is $\bar\gamma_0\bar\nabla_0= \bar\nabla_F\bar\gamma_0$.

The connections $\bar\nabla_i$'s on other $\E^i$'s, commuting with $\bar\gamma$ are constructed consecutively by decreasing $i$ and 
using the same argument with replacing the sequence (\ref{resol1}) by its truncation 
$$
0\to \E^{-n}\to \dots \to \E^{-i}\to \F_i\to 0,
$$
where $\F_1$ is the kernel of $\bar\gamma_0$ and $\F_i$, for $i\ge 2$, is the kernel of $\bar\gamma:  \E^{-i+1}\to  \E^{-i+2}$.

Using the Leibniz rule, we extend $\bar\nabla$ to a differential operator (also denoted by
$\bar\nabla$) on $M^{\sharp}=\A\otimes \E^{\bu}$ which has bidegree $(1,0)$ 
and (anti)commutes with $\bar\gamma$.

We want to find a differential in $M^{\sharp}$ of the form ${\bar D}=\bar\gamma +\bar\nabla + \sum_{i\ge 2} \bar\beta_i,$
where $\bar\beta_i$'s are $\A^0$-module endomorphisms of $M$ of bidegree $(i,1-i)$.
Recall that the equation ${\bar D}^2=0$ implies a series of equations (\ref{D2=0})
on components of ${\bar D}$.

We proceed by induction. Suppose we have already found $\bar\beta_i$'s, for $i=2, \dots, k-1$, which satisfy the
first $k-1$ equations of (\ref{D2=0}). Then the $k$-th equation looks like:
\begin{equation}\label{eqk}
[\bar\gamma, \bar\beta_k]+u_k=0,
\end{equation}
where $u_k$ depends on $ \bar\nabla, \bar\beta_2,\dots ,\bar\beta_{k-1}$. 

Note that $u_k\in \Gamma (X, \A^k\otimes_{\A^0}{\cal E}nd^{2-k}_{\A^0}\E^{\bu})$ and Bianchi identity implies $[\bar\gamma, u_k]=0$. 

The exact sequence (\ref{resol1}) can be interpreted as a quasi-isomorphism of $\E^{\bu}$ with $\F$, which implies the isomorphism in $\D^b(\A^0\mathrm{-\mod})$:
\begin{equation}\label{ende-endf}
{\cal E}nd^{\bu}_{\A^0}\E^{\bu}={\mathbb R}{\cal E}nd_{\A^0}\F .
\end{equation}
\noindent
Therefore, the $[\bar\gamma,\bu\,]$-complex ${\cal E}nd^{\bu}_{\A^0}\E^{\bu}$ has trivial cohomology in all negative degrees. 
Since $\A^k$ is a flat sheaf over $\A^0$ and all the sheaves in the complex are fine, hence acyclic,  
the complex $\Gamma (X, \A^k\otimes_{\A^0}{\cal E}nd^{\bu}_{\A^0}\E^{\bu})$ has trivial cohomology 
in all negative degrees for each $k$. The differential in this complex is exactly $[\bar\gamma,\bu\,]$. 
This implies that the equation (\ref{eqk}) has solution for any $k\ge 3$, 
because, then, $u_k$ is a closed element of negative degree $2-k$ in this complex. 

We are left only with the 3rd equation in (\ref{D2=0}), 
where $u_2=\bar\nabla^2\in \Gamma (X, \A^2\otimes_{\A^0}{\cal E}nd^0_{\A_0}\E)$. 
The obstruction to the solution of the equation $[\bar\gamma, \bar\beta_2]+u_2=0$ 
lies in the zeroth cohomology of the complex $\Gamma (X, \A^2\otimes_{\A^0}{\cal E}nd^{\bu}_{\A^0}\E^{\bu})$, 
which, by above,  equals $\Gamma (X, \A^2\otimes_{\A^0}{\cal E}nd_{\A^0}\F )$. 
Note that $\bar\nabla$ is a chain presentation for ${\bar\nabla}_F$, because $\bar\nabla$ 
commutes with $\bar\gamma$ and in view of eq.\ \eqref{nabla-nablaf}.
This implies that, under the identification of the zeroth cohomology of the complex 
$\Gamma (X, \A^1\otimes_{\A^0}{\cal E}nd_{\A^0}\E^{\bu})$ with 
$\Gamma (X, \A^1\otimes_{\A^0}{\cal E}nd^{\bu}_{\A^0}\F )$, 
the connection $\bar\nabla$ corresponds to ${\bar \nabla}_F$.  
Then, the cohomology class of $u_2=\bar\nabla^2$ is identified with ${\bar \nabla}_F^2=0$. 
Therefore, the obstruction to solving the equation $[\bar\gamma, \bar\beta_2]+u_2=0$ 
with respect to $u_2$ also vanishes. 

Now let us check that the $\dbar$-superconnection $M$ constructed in this way is quasi-isomorphic to
$F$. We have the chain map $\bar\gamma_0: M\to \A\otimes F$. The filtration
of $M$ induced by the standard filtration of $\A$ gives a spectral sequence which implies that $\bar\gamma_0$ is a
quasi-isomorphism. The fact that the cohomology of $\A\otimes F$ is $F$ can be checked
locally. Take locally a resolution $E^{\bu}$ of $F$ by free $\O_X$-modules.
Consider the bicomplex $\A\otimes E^{\bu}$. The comparison of the two spectral
sequences of the bicomplex completes the proof of Theorem \ref{m-as-superc}.
$\square$

\medskip
Since the functor $\Phi$ is fully faithful by Proposition \ref{fulfai}, it follows that its image is a full
triangulated subcategory in $\dbcoh (X)$ that contains all coherent sheaves.
Therefore, the essential image coincides with $\dbcoh (X)$. This concludes the proof
of Theorem \ref{equiv}.

\mysection{Chern Classes of Superconnections}\label{Chern}

One way to define Chern classes with values in de Rham cohomology for any coherent sheaf $F$ 
on a smooth complex manifold is to choose an appropriate real-analytic or smooth resolution 
for the sheaf $\F=\A^0\otimes_{\O_X}F$ as in Lemma \ref{lem-resol}, apply the standard construction of 
Chern-Weil to the members of the resolution and to alternate them.
According to the example of C. Voisin  \cite{Voisin}, not every coherent sheaf has a 
resolution by holomorphic vector bundles. Hence, Chern classes of the members of the resolution might not have $(p,p)$-type. 

We refine this approach via superconnections in order to establish the fact that the classes of $F$ 
does have $(p,p)$-type, and then apply this technique to defining Bott-Chern classes of coherent sheaves 
and, more generally, of objects in $\dbcoh (X)$. We define Chern classes for objects in the derived category 
$\dbcoh (X)$ by using the presentation of them by ${\bar {\partial}}$-superconnections. 
We will see that this allows us to establish 
the property of the Chern classes to be of the diagonal type in a relatively easy way.

An alternative approach to define Chern classes of coherent sheaves is, first, to get rid of the torsions by 
induction in the dimension of their supports and, then, 
to choose a suitable birational transformation 
of the manifold, which reduces the problem to the case when the coherent sheaf is locally free ({\it cf.} \cite{Griv}). 

\subsection{\sc Chern Classes of de Rham Superconnections}\label{CCSC}
We follow here the ideas of Quillen \cite{Q} and Freed \cite{F}. 
We acknowledge also expositions in \cite{BSW,Qi}

In this subsection we consider a smooth real manifold $X$ and the sheaf $\Lambda^{\bu}$ of graded algebras 
of smooth complex-valued forms on $X$.
We consider $\Lambda^{\bu}$ as a ${\mathbb Z}/2$ graded sheaf of algebras. 
Let ${\cal M}$ be a ${\mathbb Z}/2$ graded module locally free over this algebra and 
${\cal D}:{\cal M}\to {\cal M}$ an odd differential operator (not necessarily squared to 0), 
satisfying the ordinary Leibniz rule for local sections $s$ of ${\cal M}$ and $\alpha $ of $\Lambda^{\bu}$:
$$
{\cal D}(\alpha s)={\rm d}{\alpha}\cdot s +(-1)^{|\alpha |}\alpha {\cal D}(s)\,,
$$
where $\rm d$ is de Rham's differential.

We say that $({\cal M},{\cal D})$ is a {\em de Rham superconnection}.
Note the essential difference in our terminology: for de Rham superconnections we don't require ${\cal D}^2=0$, while for $\dbar$-superconnections ${\bar D}^2=0$.

Define the curvature of $({\cal M},{\cal D})$ by the formula:
$$
{\cal F}= {\cal D}^2.
$$
This is an even $\Lambda^{\bu}$-endomorphism of ${\cal M}$. There is a supertrace functional, $\tr$, 
which is defined on endomorphisms of a locally free module over a (super)commutative algebra 
$\Lambda^{\bu}$ and takes values in $\Lambda^{\bu}$.

Consider the form $\omega_k\in \Lambda^{\rm even}$ defined by the (super)trace of the $k$-th power of ${\cal F}$:
$$
{\omega}_k:= {\rm tr}{\cal F}^k.
$$
Note that $\omega_k$ is a sum of forms of various even degrees. By the same argument as for the standard Chern classes, 
we see that $\omega_k$ is a closed form. Indeed:
$$
{\rm d}({\rm tr}{\cal F}^k)={\rm tr}[{\cal D}, {\cal F}^k]=k{\rm tr}([{\cal D}, {\cal F}]{\cal F}^{k-1})=0,
$$
because of the Bianchi identity:
$$
[{\cal D}, {\cal F}]=[{\cal D}, {\cal D}^2]=0.
$$
Here and henceforth we denote $[\,\bu\,,\bu\,]$ the supercommutator with respect to the grading on forms.
Define the $k$-th coefficient of the Chern character of ${\cal M}$ via the cohomology class of $\omega_k$:
\begin{equation}\label{chernderham}
{\rm ch}_k{\cal M}=\frac 1{k!}\left({\text{\small$\frac{\mathrm i}{2\pi}$}}\right)^k\left[\omega_k\right].
\end{equation}

Any replacement of ${\cal D}$ with another odd operator ${\cal D}'$ in $\M$ satisfying the same Leibniz rule 
does not change the cohomology class of $\omega_k$.
Indeed, the set of such operators is an affine space, hence we can consider a smooth path ${\cal D}(t)$ 
connecting ${\cal D}$ with ${\cal D}'$.
Let $a(t)={\frac {\rm d}{{\rm d}t}\cal D}(t)$. Then again in view of the Bianchi identity, we have: 
$$
{\frac {\rm d}{{\rm d}t}}\omega_k(t)= k\, {\rm tr}([{\cal D}(t), a(t)]{\cal F}^{k-1}(t))=k\, {\rm d}({\rm tr}\, a(t){\cal F}^{k-1}(t)),
$$
i.e. the derivative is an exact form, which implies that the cohomology class of $\omega_k(t)$ is constant along the path.

Consider ${\cal E}={\cal M}/\Lambda^+{\cal M}$, where $\Lambda^+$ is the positive degree part of $\Lambda^{\bu}$.
${\cal E}$ is a ${\mathbb Z}/2$ graded vector bundle: ${\cal E}={\cal E}^+\oplus {\cal E}^-$. 
By choosing a splitting ${\cal E}\to {\cal M}$ of the quotient map ${\cal M}\to {\cal E}$, 
we get a non-canonical isomorphism of $\Lambda^{\bu}$-modules ${\cal M}=\Lambda^{\bu}\otimes {\cal E} $. 
We can choose ${\cal D}$ to be a direct sum of ordinary connections in ${\cal E}^+$ and ${\cal E}^-$ extended 
in the standard way to $\Lambda^{\bu}\otimes{\cal E}^+ $ and $\Lambda^{\bu}\otimes {\cal E}^-$.
It follows that the curvature of ${\cal D}$ is a direct sum of the curvatures 
of the connections on ${\cal E}^+$ and ${\cal E}^-$. 
Then, according to the Chern-Weil definition of Chern classes for vector bundles,  
the cohomology class $\frac 1{k!}(\frac{\rm i}{2\pi})^k[\omega_k]$ lies in ${\rm H}_{\rm dR}^{2k}(X, {\mathbb C})$ 
and equals to the $k$-th component of the Chern character for ${\cal E}$: 
\begin{equation}\label{chernme}
{\rm ch}_k{\cal M}={\rm ch}_k{\cal E}^+-{\rm ch}_k{\cal E}^-.
\end{equation}
Note that the right hand side is nothing but the ${\rm ch}_k\E$ for a $\ZZ/2$-graded vector bundle $\E$. 
It is well known that  ${\rm ch}_k{\E}$ and, hence, ${\rm ch}_k\M$ belong in fact to $H^{2k}_{\rm dR}(X,\QQ)$.

\subsection{\sc Chern Classes of ${\bar {\partial}}$-Superconnections} \label{52}

Now let $X$ be a complex manifold and $ M$ a flat $\bar\partial$-superconnection with differential ${\bar D}_M:M\to M$ on $X$. 
Consider the quotient ${\cal E}=M/{\cal A}^{0,+}\!M$. It has a ${\mathbb Z}$ grading: ${\cal E}={\cal E}^{\bu}$.  

Let us consider $X$ as a real manifold and replace $M$ by a module ${\cal M}$ 
over the algebra of smooth complex-valued forms bigraded in a standard way by their $(p,q)$-types 
${\Lambda}^{\bu}={\cal A}^{\bu , \bu}$:
$$
{\cal M}={\cal A}^{\bu, \bu}\otimes_{{\cal A}^{0,{\bu}}}M
$$

Define the operator ${\bar D}: {\cal M}\to {\cal M}$ as an extension of ${\bar D}_M$ from $M$ to ${\cal M}$ 
by the following Leibniz rule for local forms $\alpha\in {\cal A}^{\bu, \bu}$ and local sections $s$ of $ M$:
$$
{\bar D}(\alpha  s)={\bar {\partial}}\alpha \cdot s +(-1)^{|\alpha |}\alpha {\bar D}_M(s).
$$

We define the {\em Chern character of a ${\bar {\partial}}$-superconnection} $M$ 
as the Chern character of an arbitrary de Rham superconnection in 
${\cal M}$, that is ${\rm ch}_k(M)={\rm ch}_k(\M)$. 
We will prove that the $\mathrm{ch}_k$-component of 
this character has a $(k,k)$-type, what means (on a not necessarily K\"ahler manifold) 
a de Rham class which has a representative by a $(k,k)$-form.

A choice of splitting ${\cal E}^{\bu} \to M$ for the factorization map $M\to {\cal E}^{\bu}$ 
defines a non-canonical ${\cal A}^{\bu}$ module isomorphism 
$M={\cal A}^{\bu}\otimes_{{\cal A}^{0}}{\cal E}^{\bu}$ and an ${\cal A}^{\bu,\bu}$ 
module isomorphism: 
\begin{equation}\label{splitting_of_M}
{\cal M}={\cal A}^{\bu ,\bu }\otimes_{{\cal A}^{0}}{\cal E}^{\bu}.
\end{equation}
We shall consider $\M$ as a $\ZZ/2$-graded module $\M=\M_+\oplus\M_-$\,, where
\begin{equation}
    \M_+=\!\!\!\bigoplus_{p+q+i=\in2\ZZ}\A^{p,q}\otimes\E^i\,,~~~~
    \M_-=\!\!\!\bigoplus_{p+q+i\in2\ZZ+1}\A^{p,q}\otimes\E^i
\end{equation}

Let us chose arbitrary Hermitian forms $h_i$'s, on each vector bundle ${\cal E}^i$. 
We denote $(\,\bu\,,\bu\,)$ the sesquilinear pairing defined by $h_i$'s on ${\cal E}^{\bu}=\oplus\,\E^i$.
We use the same notation for its extension to a graded ${\cal A}^{\bu,\bu}$-valued 
pairing on ${\cal M}$, which is ${\cal A}^{\bu,\bu}$-sesquilinear  
with respect to the complex conjugation on forms in ${\cal A}^{\bu ,\bu }$. 
In other words, for any $s_1, s_2\in\M$, 
we get a pairing $(s_1, s_2)\in {\cal A}^{\bu ,\bu }$, 
and for any $\alpha \in {\cal A}^{\bu ,\bu }$ this pairing satisfies:
$$
(s_1, \alpha s_2)=(-1)^{|\alpha ||s_1|}\alpha (s_1, s_2),\ \ 
(\alpha s_1, s_2 )={\bar \alpha} (s_1, s_2). 
$$

Let $D$ be an operator on ${\cal M}$ uniquely defined by the equation 
on arbitrary local sections $s_1$ and $s_2$ of ${\cal M}$:
\begin{equation}\label{d-from-dbar}
{\bar {\partial}}(s_1, s_2)=(Ds_1, s_2)+(-1)^{|s_1|}(s_1, {\bar D}s_2)\,,
\end{equation}
or, equivalently:
\begin{equation}\label{d-from-dbar2}
{\partial}(s_1, s_2)=({\bar D}s_1, s_2)+(-1)^{|s_1|}(s_1, Ds_2)\,. 
\end{equation}

The existence and uniqueness of such  $D$ is similar to the existence and uniqueness 
of a Hermitian connection in a holomorphic vector bundle compatible with the 
$\bar \partial$-operator on local sections of the bundle.

Note that $D$ satisfies the Leibniz rule:
$$
D(\alpha  s)={\partial}\alpha \cdot s +(-1)^{|\alpha |}\alpha D(s)\,. 
$$

Make ${\cal M}$ into a de Rham superconnection by introducing an odd (in  
${\mathbb Z}/2$ grading) differential operator ${\cal D}$ on ${\cal M}$:
$$
{\cal D}= D+{\bar D}.
$$
Then ${\cal D}$ satisfies:
$$
{\rm d}(s_1, s_2 )= ({\cal D}s_1, s_2)+(-1)^{|s_1|}(s_1, {\cal D}s_2).
$$

Since ${\bar D}^2=0$, we have that: 
\begin{multline*}
    0={\partial }^2(s_1, s_2) = \\
    (D^2s_1, s_2)+(-1)^{|s_1|+1}(Ds_1, {\bar D}s_2)+(-1)^{|s_1|} 
    (Ds_1, {\bar D}s_2)+(s_1, {\bar D}^2 s_2) =  
    (D^2s_1, s_2) 
\end{multline*}
As this holds for any $s_1, s_2\in {\cal E}^{\bu}$, it follows that $D^2=0$. 
We can say that $D$ defines a flat ${\partial}$-superconnection on the module 
${\tilde M}={\cal A}^{\bu , 0}\otimes_{{\cal A}^{0,0}}{\cal E}^{\bu}$.

Then the curvature ${\cal F}$ of the de Rham superconnection ${\cal D}$ has the expression:
\begin{equation}\label{curvaturecald}
{\cal F}={\cal D}^2=[D, {\bar D}].
\end{equation}

Now, we fix a (non-canonical) isomorphism 
${\cal M}={\cal A}^{\bu,\bu }\otimes_{{\cal A}^{0,0}}{\cal E}^{\bu}$ 
as in eq.\ (\ref{splitting_of_M}), and decompose (not necessarily ${\cal A}^{\bu ,\bu }$-linear) 
the operators acting on sections  of ${\cal M}$ into their homogeneous 
components with respect to the \textit{triple degree}, 
where an operator with degree $(p,q,r)$ acts as 
${\cal A}^{a,b}\otimes_{{\cal A}^{0,0}}{\cal E}^{c}\to{\cal A}^{a+p,\,b+q}\otimes_{{\cal A}^{0,0}}{\cal E}^{c+r}$,  
that is, it shifts the grading as: 
$$
(a,b,c)\to (a+p, b+q, c+r).
$$
We denote $\deg (A)=(d_1(A),d_2(A),d_3(A))$ the triple degree of a homogeneous operator $A$. 
Similarly, the \emph{tridegree} of a section $s$ of ${\cal A}^{a,b}\otimes_{{\cal A}^{0,0}}{\cal E}^{c}$ 
is denoted as 
$$
\deg (s)=(d_1(s),d_2(s),d_3(s)):=(a,b,c).
$$
Also, for a form $\omega$ in ${\cal A}^{a,b}$, we will use the notation for its \emph{bidegree}:
$$
\deg (\omega )= (d_1(\omega ), d_2(\omega )):=(a,b).
$$
Since the pairing $(\,\bu\,,\bu\,)$ on $\cal M$ is sesquilinear, 
for any two homogeneous sections $s_1$ and $s_2$ of $\cal M$ we have that
\begin{equation}\label{weight-of-form}
    \deg((s_1, s_2))=
    (d_2(s_1)+d_1(s_2), d_1(s_1)+d_2(s_2))
\end{equation}

To simplify notation, we denote ${\bar \beta_0}:={\bar \gamma}$ and ${\bar \beta_1}:={\bar \nabla}$ 
the components in the expression \eqref{expansion} for superconnection ${\bar D}$. 
Then, the degree decomposition for $\bar D$ reads:
\begin{equation}\label{dbar-via-beta}
{\bar D}=\sum_{i\ge 0}{\bar \beta_i}\,,
\end{equation}
where ${\bar \beta_i}$ is homogeneous of degree 
\begin{equation}\label{degree-bar-beta}
\deg ({\bar \beta_i})=(0,i, 1-i).
\end{equation}

\begin{lemma}\label{degrees}
The operator $D$ has the following degree decomposition:
\begin{equation}\label{d-via-beta}
D=\sum_{i\ge 0}\beta_i,
\end{equation}
where $\beta_i$ has degree $(i,0, i-1)$.
\end{lemma}

\noindent
\textit{Proof.}
Let $\omega_{(P,Q)}$ denote the bidegree $(P,Q)$ component of an inhomogeneous 
form $\omega$ in $\A^{\bu,\bu}$. Consider 
a decomposition of $D$ in its tridegree components,
$$
D=\sum D^{p,q,j}.
$$
Let us choose homogeneous local sections $s_1,s_2$ of $\M$, and denote $\deg((s_1,s_2))=(A,B)$. 
Formula \eqref{d-from-dbar} implies that
\begin{equation}\label{dpqr}
    \sum_{p,q,i}(D^{p,q,j}s_1, s_2)_{(P,Q)}=\bar\partial(s_1,s_2)_{(P,Q)}-
    \sum_i(-1)^{|s_1|}(s_1,\bar\beta_is_2)_{(P,Q)}\,.
\end{equation}
Then, we notice that all terms in the right hand side vanish unless $P=A$ and 
$Q=B+d_3(s_2)-d_3(s_1)+1$. Indeed, note that $(s_1,s_2)=0$ unless $d_3(s_1)=d_3(s_2)$, hence $\deg(\bar\partial(s_1,s_2))=(A,B+1)=(A, B+d_3(s_2)-d_3(s_1)+1)$. Also, $(s_1,\bar\beta_is_2)=0$ if 
$d_3(s_1)\neq d_3(\bar\beta_is_2)$, while 
$d_2(s_1,\bar\beta_is_2)=B+i$ and 
$d_3(\bar\beta_is_2)=d_3(s_2)+1-i$, which implies the same formula the degree of the last term in (\ref{dpqr}). 

Since $\deg((D^{p,q,j}s_1, s_2))=(A+q,B+p)$
and $d_3(D^{p,q,j}s_1)=d_3(s_1)+j$, 
we have that 
$(D^{p,q,j}s_1, s_2)_{(P,Q)}=0$ unless $q=0$, $p=d_3(s_2)-d_3(s_1)+1$, 
and $d_3(s_1)+j=d_3(s_2)$. Since $s_1,s_2$ are arbitrary, we conclude that 
$D^{p,q,j}=0$ unless $q=0$ and $p=j+1$, which proves the lemma.
$\square$

\medskip
In view of the formulas (\ref{curvaturecald}), (\ref{dbar-via-beta}) and (\ref{d-via-beta}), the curvature has the form:
\begin{equation}\label{curv-by-betas}
     {\cal F}=\sum_{i,j\geqslant 0}[\beta_i,\bar{\beta_j}]
\end{equation}

Decompose the Chern form $\omega_k={\rm tr}{\cal F}^k$ into $(p,q)$-components: $\omega_k= \sum \omega_k^{p,q}$.
We shall show that the form has trivial all the off-diagonal entries.
\begin{lemma}\label{pqvanishing}
We have that 
$\omega_k^{p,q}=0,\ {\rm if}\ p\ne q.$
\end{lemma}

\noindent
\textit{Proof.}~
According to (\ref{dbar-via-beta}) and (\ref{d-via-beta}),
the possible degrees for $D$ are $(p, 0, p-1)$, where $p\ge 0$, and the degrees for ${\bar D}$ are $(0, q, -q+1)$, where $q\ge 0$.
Then possible degrees for ${\cal F}=[D, {\bar D}]$ and for its powers ${\cal F}^k$ are $(p, q, p-q)$, for some $p,q\ge 0$. 

Note that the trace of an endomorphism is zero on components with degrees $(p,q,r)$ if $r\ne 0$, hence the result.
$\square$

\medskip
Now, we vary the Hermitian forms in vector bundles ${\cal E}^j$'s with a parameter $t$ according to the rule:
\begin{equation}\label{ht}
    h_j(t)= t^ih_j\,.
\end{equation}

This defines a variation   of the de Rham superconnection, ${\cal D}(t)$, and its curvature, ${\cal F}(t)$. 
Let us see how the Chern form $\omega_k(t)={\rm tr}{\cal F}(t)^k$ varies. 
Decompose the form $\omega_k(t)$ into its $(p,p)$-components 
with account of the result of Lemma \ref{pqvanishing}: 
$
\omega_k(t)= \sum_{p\ \ge 0} \omega_k^{p,p}(t).
$
In particular, $\omega_k=\omega_k(1)= \sum \omega_k^{p,p}$\,.

\begin{lemma}\label{lemmaht}
The following equality holds: ~
$
\omega_k^{p,p}(t)=t^{p-k}\omega_k^{p,p}.
$
\end{lemma}

\noindent
\textit{Proof.}~ 
Note that ${\bar D}$ does not change under the variation of metrics.  
The operator $D(t)$ is defined by eq.\ \eqref{d-from-dbar} via $\bard$ and $h_i(t)$ and we are going to 
determine how it depends on $t$.
To this end we extract various homogeneous components of eq.\ \eqref{d-from-dbar} as we did 
in proving Lemma \ref{degrees}, \textit{cf.}, eq.\ \eqref{dpqr}. 

First, we take $s_1,s_2$, such that $d_3(s_1)=d_3(s_2)=j$. Then, the hermitian form $(s_1,s_2)$ 
is determined by the metric $h_j(t)$ in only one component, $\E^j$, of $\E^{\bu}$ and we write it as 
$(s_1,s_2)(t)=t^j(s_1,s_2)_j$ in agreement with eq.\ \eqref{ht}. 
We consider now the following component of eq.\ \eqref{d-from-dbar},
\begin{align*}
    \dbar(s_1, s_2)(t)=(\beta_1(t)s_1, s_2)(t)+(-1)^{|s_1|}(s_1,\barb_1s_2)(t)\implies \\
    t^j\dbar(s_1, s_2)_j=t^j(\beta_1(t)s_1, s_2)_j+(-1)^{|s_1|}t^j(s_1,\barb_1s_2)_j \,.
\end{align*}
We conclude from this relation, that $\beta_1$ does not depend of $t$, which is part of the 
assertion of the lemma. 

Second, to find $\beta_i(t)$ for $i\neq 1$, we take $d_3(s_2)=j$, $d_3(s_1)=j+1-i$.
In this case, the term in the left hand side of eq.\ \eqref{d-from-dbar} vanishes and
we are left with the following relation
\begin{align*}
    (\beta_i(t)s_1, s_2)(t)+(-1)^{|s_1|}(s_1,\barb_is_2)(t)= \\ 
    t^j(\beta_i(t)s_1, s_2)_j+(-1)^{|s_1|}t^{j+1-i}(s_1,\barb_is_2)_{j+1-i}=0 \,.
\end{align*}
Thus, we conclude that, for all $i$, 
$$
\beta_i(t)=t^{i-1}\beta_i \,. 
$$
We turn now to $\omega_k^{p,p}$ and notice that, by definition,
$$
    \omega_k^{p,p}(t)=\sum\prod_{s=1,\ldots,k}[\beta_{p_s}(t),\barb_{q_s}]  \,,
$$
and the sum runs over $k$-tuples of pairs $(p_1,q_1),\ldots,(p_k,q_k)$, such that 
$\sum p_s=p$ and $\sum q_s=q$.

Then, we observe, that $\omega_k^{p,p}(t)=t^N\omega_k^{p,p}$\,, where $N=\sum_{s=1}^k(p_s-1)=p-k$, 
whence the result. $\square$

Let $H^{p,p}_{dR}(X)$ denote the de Rham cohomology classes which can be represented by a closed 
$2p$-form of type $(p,p)$.
Since the cohomology class of $\omega_k$ does not depend on connection, it cannot in particular vary 
with $t$ when the Hermitian structure varies according to eq.\ \eqref{ht}. 
Then, Lemmas \ref{pqvanishing} and \ref{lemmaht} imply the following
\begin{prop}\label{kth}
The $k$-th Chern character of a $\dbar$-superconnection is of type $(k,k)$, that is, 
${\rm ch}_k(M)\in H^{k,k}_{dR}(X)$.
\end{prop}

\begin{rem}
As a matter of fact, the above Proposition follows immediately from 
Lemma \ref{pqvanishing} alone and the general properties of 
the Chern character of a superconnection described in Subsection \ref{CCSC}, 
namely that the $k$-th character ${\rm ch}_k\M$ belongs to $H^{2k}_{\rm dR}(X,\QQ)$, 
\textit{cf.}, eq.\,\eqref{chernme}. 
Thus, in principle, we do not need Lemma \ref{lemmaht} to prove the last Proposition \ref{kth}. 
However, the technical result of that Lemma will be necessary below in the next Subsection.
\end{rem}

\subsection{\sc The Chern Character of a Coherent Analytic Sheaf}

We are going to show that the definition of a Chern character given above in Subsection \ref{52} 
descends to $\dbcoh (X)$ via the equivalence of categories $\dbcoh (X)\cong\Ho(\C_X)$\,. First, recall
\begin{lemma}\label{qii}
Let $\E^{\bu}_1\xrightarrow{\varphi}\E^{\bu}_2$ be a morphism of complexes of smooth vector bundles 
(locally free $\A^{0}$-modules in terminology of the present paper). If $\varphi$ is a quasi-isomorphism then 
$\ch(\E^{\bu}_1)=\ch(\E^{\bu}_2)$, where $\ch(\E^{\bu})=\sum(-1)^i\ch(\E^i)$.
\end{lemma}

\noindent
\textit{Proof.}~
Denote $\gamma_1,\,\gamma_2$ the differentials in $\E^{\bu}_1$ and $\E^{\bu}_2$, respectively, and 
consider the cone of $\varphi$ realised as a complex $C^i=\E^{i+1}_1\oplus\E^i_2$ with a differential, 
which maps $(a,b)\in C^i$ to $(-\gamma_1(a)\,,\,\varphi(a)+\gamma_2(b))\in C^{i+1}$. Obviously, 
$\ch(C^{\bu})=\ch(\E^{\bu}_2)-\ch(\E^{\bu}_1)$. On the other hand, if $\varphi$ is a quasi-isomorphism, 
the complex $C^{\bu}$ is acyclic and, hence, $\ch(C^{\bu})=0$, whence the statement of the lemma.
$\square$

\medskip
Let $M_1$ and $M_2$ be two objects in the category $\C_X$, that is two $\dbar$-superconnections on $X$. Let 
$\E_\alpha:=M_\alpha/\A^+M_\alpha\,,\;\alpha=1,2$. These are complexes of vector bundles.
Suppose now that $M_1$ and $M_2$ are isomorphic in the 
category $\Ho(\C_X)$. This implies (but not equals to) that there is a quasi-isomorphism 
$M_1\xrightarrow{\varphi}M_2$\,. This quasi-isomorphism descends to a quasi-isomorphism 
$\E^{\bu}_1\xrightarrow{\varphi}\E^{\bu}_2$\,. On the other hand, 
$\ch(M_\alpha)=\ch(\E_\alpha)$ as it has been defined in eq.\ \eqref{chernme}. 
Hence, by Lemma \ref{qii},
$\ch(M_1)=\ch(M_2)$ and we conclude that the Chern character 
depends only on the isomorphism class in $\Ho(\C_X)$. By equivalence of categories it means that 
the Chern character is also defined on the isomorphism classes of objects in $\dbcoh (X)$. 
This gives, in particular, a definition of the Chern character for coherent sheaves on $X$. 

\subsection{\sc Bott-Chern Cohomology}
We have seen above that, given a $\ZZ/2$-graded module $\mathcal{M}$ over $\Lambda^{\bu}=\A^{\bu,\bu}$ 
with a superconnection $\D$ on it, 
one constructs de Rham's cohomology classes ${\rm ch}_k(\mathcal{M})=\left[\frac{1}{k!}\mathrm{tr}\left(\frac{\rm i}{2\pi}\mathcal{F}\right)^k\right]\in 
H^{2k}_{\rm dR}(X,\QQ)$, where $\mathcal{F}=\D^2$ is the curvature and ${\rm ch}_k(\mathcal{M})$ does not 
depend on the superconnection, but only on $\mathcal{M}$ itself. Let us now show that if $(\M,\D)$ comes from 
a $\dbar$-superconnection $(M,\bar D)$ as in Subsection \ref{52}, one can define ${\rm ch}_k(M)$ 
as an element in more refined cohomology. 

\mathchardef\mhyphen="2D
Let us consider the following cohomology groups for a complex-analytic manifold $X$,
\begin{equation}
    H^{p,p}(X)
    :=\frac{\{\mathrm{complex}~d\mathrm{\mhyphen closed}~(p,p)\mathrm{\mhyphen forms~on}~X\}}{\{d\mathrm{\mhyphen exact~forms}\}}
\end{equation}
and
\begin{equation}\label{BC}
    H^{p,p}_{\mathrm BC}(X)
    :=\frac{\{\mathrm{complex}~d\mathrm{\mhyphen closed}~(p,p)\mathrm{\mhyphen forms~on}~X\}}{\{\dbar\partial\mathrm{\mhyphen exact~forms}\}} \,.
\end{equation}
The latter are known as Bott-Chern cohomology, and we have obviously a surjection 
$H^{p,p}_{\mathrm BC}(X)	\twoheadrightarrow H^{p,p}(X)$. On a K\"ahler manifold, 
by Hodge theory and $\partial\dbar$-lemma, we have that $H^{p,p}_{\mathrm BC}(X)=H^{p,p}(X)=H^p(X,\Omega^p_X)$, 
which is not true in general, cf., \cite{AngT}.

In Subsection \ref{52} we saw that for a $\dbar$-superconnection $M$, $\mathrm{ch}_k(M)\in H^{k,k}(X)$. 
It is known that, for a holomorphic bundle $E$, its Chern classes, or character can in fact be defined as elements of 
Bott-Chern cohomology \eqref{BC}. Let us show that the same is true for $\dbar$-superconnections and, hence, also 
for coherent sheaves on complex-analytic manifolds.

Let $(M,\bar D)$ be a $\dbar$-superconnection on $X$. Recall 
(Subsection \ref{52}) that in order to 
define $\mathrm{ch}(M)$ we chose a splitting
\begin{equation}\label{splt}
{\cal M}:={\cal A}^{\bu,\bu}\otimes_{{\cal A}^{0, \bu}}M\toiso 
\A^{\bu,\bu}\otimes_{\A^0}\E^{\bu} \,,
\end{equation}
and hermitian metrics on $\E^i$'s. Then we get in $\M$ a de Rham superconnection $\D=D+\bar D$, where 
$D$ depends on the metric chosen and satisfies (cf., eq.\ \eqref{d-from-dbar})
\begin{equation}\label{d-from-dbar3}
    \dbar\,(\phi,\psi)=(D\phi,\psi)+(-1)^{|\phi|}(\phi,{\bar D}\psi)\,.
\end{equation}
Here $\phi$ and $\psi$ are sections of $\M$, and 
$(\;,\,)$ is the $\A^{\bu,\bu}$-valued pairing defined by the metric in $\E^{\bu}$.

\newcommand{\deh}{\delta h}

Let $\F=\D^2$, consider the Chern forms $\omega_k=\mathrm{tr}\F^k$ and 
discuss how everything depends on the choices made, \textit{i.e.} the splitting \eqref{splt} and 
the metric in $\E^{\bu}$. 
A change of the splitting is the same as an $\A^{\bu,\bu}$-automorphism of 
$\A^{\bu,\bu}\otimes_{\A^0}\E^{\bu}$, which is the same as a strict gauge transformation. 
The forms $\omega_k$ are obviously invariant upon this. 
Let us now consider an infinitesimal variation of the metric. It can be described by a hermitian 
section $\deh$ of $\E nd\,\E^{\bu}$ as
$
\delta (\phi,\psi)=(\phi,\deh\,\psi)=(\deh\,\phi,\psi) \,.
$
The infinitesimal variation of \eqref{d-from-dbar3} reads (note, that $\bar D$ does not change with $\deh$):
\begin{equation}\label{d1}
    \dbar\,(\deh\,\phi,\psi)=(\deh D\phi,\psi)+(-1)^{|\phi|}(\deh\,\phi,{\bar D}\psi) 
    + (\delta D\,\phi,\psi) \,.
\end{equation}
On the other hand, eq.\ \eqref{d-from-dbar3} with $\phi$ replaced by $\deh\,\phi$ gives
\begin{equation}\label{d2}
    \dbar\,(\deh\,\phi,\psi)=(D\deh\,\phi,\psi)+(-1)^{|\phi|}(\deh\,\phi,{\bar D}\psi)\,.
\end{equation}
By comparing \eqref{d1} and \eqref{d2}, we obtain that $\delta D=D\deh-\deh D$. 

Since $D^2=\bar D^2=0$, the curvature of $D+\bar D$ equals $\F=[\bar D,D]$ and we get 
$\delta\F=[\bar D,[D,\deh]]$. Then, recalling that $[D,\F]=[\bar D,\F]=0$, we obtain that
$$
    \delta \mathrm{tr}\F^k=k\,\mathrm{tr}\,\delta\F\,\F^{k-1}= 
    k\,\mathrm{tr}\,[\bar D,[D,\deh]]\,\F^{k-1}= 
    k\dbar\partial\,\mathrm{tr}\,\deh\,\F^{k-1} \,.
$$
This shows that $\delta\omega_k$ is $\dbar\partial$-exact and, thus, the closed form 
$\omega_k$ gives a well defined class in $H^{k,k}_{\mathrm BC}(X)$. 
Note, that \textit{a priori} the form $\omega_k$ has components $\omega_k=\sum_{p=0}^k\omega_k^{(p,p)}$, 
where $\omega_k^{(p,p)}$ is a $(p,p)$-form. However, the metric-rescaling argument 
in Subsection \ref{52} (\textit{cf.}, Lemma \ref{lemmaht}) together 
with the present argument show that $\omega_k^{(p,p)}$\,'s are $\dbar\partial$-exact for $p<k$.

As a result, we get the following theorem.
\begin{theo}
The Chern character of a $\dbar$-superconnection $M$ is well-defined as an element 
in Bott-Chern cohomology
$$
\mathrm{ch}_k(M)\in H^{k,k}_{\mathrm BC}(X)\cap H^{2k}_{\rm dR}(X,\QQ) \,.
$$
\end{theo}

\appendix
\section*{\large Appendix}

\stepcounter{section}
\renewcommand{\C}{{\mathcal C^\infty}}
\newcommand{\Ca}{{\mathcal C^\omega}}
\newcommand{\CR}{{\mathcal C^\infty_\RR}}
\newcommand{\CaR}{{\mathcal C^\omega_\RR}}
\newcommand{\ff}{faithfully flat}
\newcommand{\n}{n$^\circ$}

In this Appendix we will prove, for the reader's convenience, that, on a smooth complex manifold $X$ of dimension $n$,
the sheaf of rings $\A^0_X$ is flat over the sheaf of rings $\O_X$ of holomorphic functions on $X$, that is, the functor
$F\mapsto F\otimes_{\O_X}\A^0_X$ is exact on the category of sheaves of $\O_X$-modules.
For sheaves, flatness is only the property of local rings at each point of $X$. For that reason, we pass
to local rings. Let us choose an arbitrary point $x\in X$ and denote $\O=\O_{X,x}$ the local ring of holomorphic
functions, $\C=\A^0_X$ the local ring of complex valued smooth functions, and
$\Ca$ the local ring of complex-valued real-analytic functions at $x$. We prove below
that $\C$ is flat over its subring $\O\subset\C$. By definition, this means that for any short
exact sequence of $\O$-modules,
$$
0\to M_1\to M_2\to M_3\to 0 \,,
$$
the corresponding sequence
$$
0\to M_1\otimes_\O\C\to M_2\otimes_\O\C\to M_3\otimes_\O\C\to 0
$$
is exact as well. 
Our proof will result from a recollection of results in literature, the most important and strongest part 
of which are due to Malgrange (\cite{Malgrange}, and in more detail below).

Actually we shall prove a stronger property of the pair $\O\subset\C$. All the rings
and ring homomorphisms below are assumed to be unital.

\begin{defn} \label{def:ff}
Let $E$ be a module over a commutative ring $A$. Then $E$ is called \ff\ if the following
two conditions hold: \\
(i) ~$E$ is $A$-flat; \\
(ii) for any $A$-module $M$, the equality $E\otimes_AM=0$ implies $M=0$ \,.
\end{defn}
In other words, the functor $M\mapsto E\otimes_AM$ on $A$-modules is (i) exact and
(ii) faithful. The following useful properties of flat algebras can be found, for
example, in \cite[Exercise 16 in Ch.3]{AM}.

\begin{prop} \label{ex16}
Let $\rho:A\to B$ be a homomorphism of rings.
Suppose $B$ is flat as an $A$-module (one says $B$ is a flat $A$-algebra). Then, the following
propositions are equivalent: \\
(i)~~ $B$ is \ff\ over $A$\,; \\
(ii)~ for any ideal ${\mathfrak a}\subset A$, one has that
$\rho^{-1}(\rho({\mathfrak a})B)={\mathfrak a}$\,; \\
(iii) the induced map $Spec\,B\to Spec\,A$ is surjective; \\
(iv) for any maximal ideal ${\mathfrak m}\subset A$, one has that
$\rho({\mathfrak m})B\neq B$ \,; \\
(v)~ for any $A$-module $M$, the map $x\mapsto x\otimes 1: M\to M\otimes_A B$ is injective.
\end{prop}

Another convenient necessary and sufficient condition of faithful flatness is given
by the following proposition taken from \cite[\S 4 of Ch.III]{Malgrange}.

\begin{prop}
Given a ring and a subring, $A\subset B$, the ring $B$ is \ff\ over $A$ if and only if
the $A$-module $B/A$ is flat.
\end{prop}
\begin{proof}{}
Let us assume first that the $A$-module $B/A$ is flat. It follows then from the short
exact sequence
\beq \label{B/A}
0\to A\to B\to B/A\to 0
\eeq
that $B$ is $A$-flat as well. Let us tensor the sequence (\ref{B/A}) with an arbitrary
$A$-module $M$:
$$
0\to M\to M\otimes_A B\to M\otimes_A B/A\to 0 \,.
$$
Since $B/A$ is flat, the sequence remains exact, which implies condition {\it (ii)}
of Definition \ref{def:ff} (cf.\ also {\it (v)} of Proposition \ref{ex16}).

Let us now assume that $B$ is a \ff\ $A$-algebra. Choosing an arbitrary short exact
sequence of $A$-modules and tensoring it with the sequence (\ref{B/A}) we obtain the
following commutative diagram:
$$
\begin{array}{ccccccccc}
&&0&&0&&0&& \\
&&\downarrow &&\downarrow &&\downarrow && \\
0&\to&M_1&\to&M_2&\to&M_3&\to&0 \\
&&\downarrow &&\downarrow &&\downarrow && \\
0&\to&M_1\otimes_A B&\to&M_2\otimes_A B&\to&M_3\otimes_A B&\to&0 \\
&&\downarrow &&\downarrow &&\downarrow && \\
0&\to&M_1\otimes_A B/A&\to&M_2\otimes_A B/A&\to&M_3\otimes_A B/A&\to&0 \\
&&\downarrow &&\downarrow &&\downarrow && \\
&&0&&0&&0&&
\end{array}
$$
The first line, $0 \to M_1 \to M_2 \to M_3 \to 0$, is exact by the assumption made.
The second line is exact, because we assume that $B$ is $A$-flat. The columns are
exact, which results from the faithful flatness of $B$ and Proposition
\ref{ex16}{\it (v)}. Altogether, this implies exactness of the third line, whence
the $A$-flatness of $B/A$.
\end{proof}

The flatness obeys the transitivity property. Thus, for three
nested rings $A\subset B\subset C$, if $B$ is $A$-flat and $C$ is $B$-flat, then $C$
is $A$-flat. Similar result holds for faithful flatness:

\begin{lemma} \label{trans}
If, for three rings $A\subset B\subset C$, $B$ is \ff\ over $A$, and $C$ is \ff\ over
$B$, then $C$ is \ff\ over $A$.
\end{lemma}
\begin{proof}{}
Since the ordinary flatness is inherited by transitivity, the statement follows
from \ref{ex16}{\it (iii)}. Alternatively, consider the short exact sequence of
$A$-modules
$$
~~~~~~~~~~~~~~~~~~~~~~~~~~ 0\to B/A\to C/A\to C/B\to 0  \,. ~~~~~~~~~~~~~~~~~~~~~(*)
$$
We know that $B/A$ and $C/B$ are $A$-flat. This implies that $C/A$ is $A$-flat.
\end{proof}

The following lemma is a slightly weakened version of \cite[Proposition 4.7]{Malgrange}.
\begin{lemma} \label{qtrans}
If, for three rings $A\subset B\subset C$, $C$ is \ff\ over $A$, and $C$ is \ff\ over $B$,
then $B$ is \ff\ over $A$.
\end{lemma}
\begin{proof}{}
In the short exact sequence $(*)$, we know that $C/A$ and $C/B$ are $A$-flat. This
implies the $A$-flatness of $B/A$.
\end{proof}

Our further argument will be based on the results of Malgrange \cite{Malgrange} and,
in the first place, on the following strong result
\cite[\S 4 of Ch.III, Theorem 1.1 and Corollarry 1.12 of Ch.VI]{Malgrange}.
\begin{theo} \label{Malgrange}
The local ring of smooth functions $\C$ is \ff\ over its subring $\Ca$, the local
ring of real-analytic functions.
\end{theo}
This theorem  is actually formulated in \cite{Malgrange} for real-valued
functions. It can however be readily extended to complex-valued functions (see
\cite{AH2}). Let $\CR$ and $\CaR$ denote the local rings of real-valued smooth and
real-analytic functions respectively. The theorem of Malgrange asserts that $\CR$ is
\ff\ over $\CaR$ and, thus, $\CR/\CaR$ is $\CaR$-flat. For any homomorphism of rings
$A\to B$ and any flat $A$-module $M$, the $B$-module $B\otimes_AM$ is $B$-flat.
Therefore, the module $(\CR/\CaR )\otimes_\CaR\Ca$ is $\Ca$-flat. On the other hand, we
find that
$$
(\CR/\CaR )\otimes_\CaR\Ca=(\CR/\CaR )\otimes_\CaR(\CaR\otimes_\RR\CC)=
(\CR/\CaR )\otimes_\RR\CC=\C/\Ca \,,
$$
whence $\C/\Ca$ is $\Ca$-flat, which shows that $\C$ is \ff\ over $\Ca$.

In view of transitivity, it remains to prove that the local ring of real-analytic functions, $\Ca$, is \ff\
over its subring of holomorphic functions, $\O$. Flatness and faithful flatness agree 
with completions of Noetherian rings. Namely (see
\cite[Proposition 10.14]{AM}, \cite[Theorem 3(iii) in Ch.III, \S 3, \n4]{B}):
\begin{prop} \label{completion}
Let $A$ be a Noetherian ring, choose any ideal ${\mathfrak a}\subset A$, and denote
$\hat A$ the ${\mathfrak a}$-adic completion of $A$. Then, $\hat A$ is a flat $A$-algebra.
\end{prop}
In the case of local Noetherian rings we have more (see
\cite[Theorem 4.9 in Ch.III]{Malgrange} and \cite[Ch.III, \S 3, \n5]{B}).
\begin{prop} \label{local completion}
Let $A$ be a local Noetherian ring with maximal ideal ${\mathfrak m}\subset A$ and
denote $\hat A$ its ${\mathfrak m}$-adic completion. Then, $\hat A$ is a \ff\ $A$-algebra.
\end{prop}
\begin{proof}{}
This follows from Propositions \ref{completion} and \ref{ex16}{\it (iv)}.
\end{proof}

The following proposition can be found in \cite[Exercise 12 in Ch.10]{AM}.
\begin{prop} \label{formal=>ff}
Let $A$ be a Noetherian ring and $B=A[[x_1,\ldots,x_n]]$ the $A$-algebra of formal
power series in $n$ variables. Then $B$ is \ff\ over $A$.
\end{prop}
\begin{proof}{}
The ring $B$ can be viewed as an ${\mathfrak a}$-adic completion of the ring of polynomials
$A[x_1,\ldots,x_n]$, where ${\mathfrak a}=(x_1,\ldots,x_n)$ is the ideal of
polynomials with zero a constant term. The flatness of $B$ over $A$ is then implied by
Proposition \ref{completion}. The faithful flatness follows now by Proposition
\ref{ex16}{\it (iv)}.
\end{proof}

We need some properties of the rings $\O$ and $\Ca$. Both are local
Noetherian rings \cite[Theorem 3.8 of Ch.III]{Malgrange}. Let now
$\hat\O=\CC[[z_1,\ldots,z_n]]$ be the ring of formal power series in $n$ variables
(the completion of the local ring $\O$) and
$\hat\Ca=\CC[[z_1,\ldots,z_n,\bar z_1,\ldots,\bar z_n]]$ the ring of formal power
series in $2n$ variables (the completion of the local ring $\Ca$). Both rings,
$\hat\O$ and $\hat\Ca$, are local Noetherian rings as well. Moreover,
\begin{prop} \label{3x2}
The following properties hold: \\
(i)~~ $\hat\O$ is \ff\ over $\O$; \\
(ii)~ $\hat\Ca$ is \ff\ over $\Ca$; \\
(iii) $\hat\Ca$ is \ff\ over $\hat\O$.
\end{prop}
\begin{proof}{}
{\it (i)} and {\it (ii)} follow by \ref{local completion}, while {\it (iii)} follows from
the isomorphism $\hat\Ca\simeq\hat\O[[\bar z_1,\ldots,\bar z_n]]$ and \ref{formal=>ff}.
\end{proof}

\begin{prop}
The ring $\Ca$ is \ff\ over its subring $\O$.
\end{prop}
\begin{proof}{}
The embeddings of rings $\O\subset\hat\O\subset\hat\Ca$ show that $\hat\Ca$ is \ff\
over $\O$ by Proposition \ref{3x2}{\it (i)} and {\it (iii)}, and by transitivity
\ref{trans}. The result follows then from the embeddings 
$\O\subset\Ca\subset\hat\Ca$ by Proposition \ref{3x2}{\it (ii)} and Lemma \ref{qtrans}.
\end{proof}

The last proposition together with the theorem of Malgrange (Proposition \ref{Malgrange})
implies now by the transitivity argument \ref{trans}:
\begin{theo}\label{flat-over-o}
    The ring $\C$ is \ff\ over its subring $\O$.
\end{theo}

\begin{corll} \label{flatness}
    On a complex analytic manifold $X$,the sheaf $\A^0_X$ of rings of smooth functions is \ff\ 
    over the sheaf $\O_X$ of rings of holomorphic functions.
\end{corll}


\begin{smallbibl}{BVdBA}

\bibitem[AH1]{AH} M.\,F.\,Atiyah, F.\,Hirzebruch,
{\em Analytic cycles on complex manifolds}, \\ 
Topology {\bf 1} (1962) 25-45

\bibitem[AH2]{AH2} M.\,F.\,Atiyah, F.\,Hirzebruch,
{\em The Riemann-Roch theorem for analytic embeddings}\/,
Topology {\bf 1} (1962) 151-166.

\bibitem [AM]{AM} M.\,F.\,Atiyah, I.\,G.\,Macdonald
{\it Introduction to Commutative Algebra}, \\ 
Addison-Wesley, Reading, MA, London (1969) ix+128 pp.

\bibitem[AneTo]{AT} M.\,Anel, B.\,To\"en, \\ 
{\em D\'enombrabilit\'e des classes d'\'equivalences d\'eriv\'ees de vari\'et\'es alg\'ebriques}, \\ 
J.\,Algebraic\,Geom. {\bf 18} (2009)
257-277
[arXiv:math/0611545v3 [math.AG]]

\bibitem[AngTo]{AngT} D.\ Angella, A.\ Tomassini, 
\newblock {\em On the $\partial\dbar$-Lemma and Bott-Chern cohomology}, \\
\newblock  Invent.\,Math. {\bf 192}, no. 1 (2013) 71-81 
\newblock 	[arXiv:1402.1954 [math.DG]

\bibitem[BSW]{BSW} J.-M.\,Bismut, Shu Shen, and Zhaoting Wei, \\ 
\newblock {\em Coherent sheaves, superconnections, and RRG} (2021) 
\newblock [arXiv:2102.08129 [math.AG]]

\bibitem[Block]{Block} J.\ Block, {\em Duality and equivalence of module categories in noncommutative geometry}. 
In A celebration of the mathematical legacy of Raoul Bott,
 CRM\,Proc., Lecture Notes, {\bf 50}, 311-339. Amer. Math. Soc., Providence, RI, 2010.

\bibitem[BK1]{BK1} A.\,Bondal, M.\,Kapranov,
\newblock {\em Representable functors, Serre functors and mutations},
\newblock Math.\,USSR Izvestiya, {\bf 35}, (1990), 519-541.

\bibitem[BK2]{BK2} A.\,Bondal, M.\,Kapranov,
\newblock {\em Enhanced triangulated categories}, \\ 
\newblock Math.\,USSR Sbornik, {\bf 70}, (1991), 93-107.

\bibitem[BLL]{BLL} A.\,Bondal, M.\,Larsen and V.\,Lunts
\newblock {\em Grothendieck ring of pretriangulated categories},
\newblock International Math.\,Research Notes, {\bf 29}, (2004), 1461-1495 \\
\href{https://arxiv.org/abs/math/0401009}{[arXiv:math/0401009[math.AG]]}

\bibitem[BO]{BO} A.\,Bondal, D.\,Orlov, \\ 
\newblock {\em Reconstruction of a variety from the derived category and groups of autoequivalences}, \\ 
\newblock Compositio Mathematica, {\bf 125}, no\,3 (2001) 327-344 \\ 
\href{https://arxiv.org/abs/alg-geom/9712029}
{[arXiv:alg-geom/9712029]}

\bibitem[BR]{BR} A.\,Bondal, A.\,Rosly,
\newblock {\em Derived categories for complex-analytic manifolds}, \\ 
\newblock IPMU11-0117, IPMU, Kashiwa, Japan  (2011) \\ 
\href{http://research.ipmu.jp/ipmu/sysimg/ipmu/672.pdf}
{[http://research.ipmu.jp/ipmu/sysimg/ipmu/672.pdf]}

\bibitem[BVdB]{BVdB} A.\,Bondal, M.\,Van den Bergh, 
\newblock {\em Generators and representability of functors
in commutative and noncommutative geometry}, 
\newblock Moscow Math.\,J, {\bf 3}, no.\,1 (2003) 1-36 \\ 
\href{https://arxiv.org/abs/math/0204218}
{[arXiv:math/0204218[math.AG]]}

\bibitem[B]{B} N.\,Bourbaki
{\it Commutative Algebra}, in: Elements of Mathematics, Hermann,
Addison-Wesley Publishing Co., Paris, Reading, MA, translated from the French
(1972) xxiv+625 pp.

\bibitem[F]{F} D.\ S.\,Freed,
{\em Geometry of Dirac Operators}, 1987 (unpublished notes)

\bibitem[G]{G} H.\,Grauert,
{\em On Levi's problem and the imbedding of real-analytic manifolds}, \\ 
Ann. of Math.\,(2) {\bf68} (1958) 460--472

\bibitem[Griv]{Griv}J.\,Grivaux,
{\em Chern classes in Deligne cohomology for coherent analytic sheaves}, \\
Math.\,Ann. \textbf{347}, No.\,2 (2010) 249-284




\bibitem[Illu]{Il} L.\,Illusie, 
{\em Existence de R\'esolutions Globals},
In: Th\'eorie des intersections et
th\'eor\`eme de Riemann-Roch, pp. 160–221. Lecture Notes in Mathematics, 
Vol. \textbf{225}, Springer-Verlag, Berlin, 1971.

\bibitem[Kap]{Kap} M.\,M.\,Kapranov,
{\em On DG-modules over the de Rham complex and the vanishing cycles functor}\/,
Algebraic geometry (Chicago, IL, 1989), Lecture Notes in Math. {\bf 1479},
Springer, Berlin (1991) 57-86

\bibitem[KS]{KS} M.\,Kashiwara, P.\,Schapira,
{\em Sheaves on manifolds}\/,
Grundlehren der Mathematischen Wissenschaften
[Fundamental Principles of Mathematical Sciences],
{\bf 292}, Springer-Verlag, Berlin, (1994) x+512 pp.

\bibitem[L]{L} J.\, Lesieutre
{\em IMRN}, no.15 (2015),  6011-6020



\bibitem[M]{Malgrange} B.\,Malgrange,
{\em Ideals of differentiable functions}\/,
Tata Institute of Fundamental Research Studies in Mathematics, No. {\bf 3},
Tata Institute of Fundamental Research, {\mbox Bombay};
Oxford University Press, London (1967) vii+106 pp.

\bibitem[O]{O}
D.O.\ Orlov, 
{\em Derived categories of coherent sheaves on abelian varieties and equivalences between them,}
Izv.\,Ross.\,Akad.\,Nauk\,Ser.\,Mat. \textbf{66}, no. 3 (2002) 131-158

\bibitem[P1]{P1}
N.\ Pali,
{\em Faisceaux  $\dbar$-coh\'erents sur les vari\'et\'es complexes,} \\
Math.\,Ann. {\bf 336} (2006) 571-615

\bibitem[P2]{P2}
N.\ Pali,
{\em Une caract\'erisation diff\'erentielle des faisceaux analytiques coh\'erents sur une vari\'et\'e complexe,}
\href{https://arxiv.org/abs/math/0301146}
{[arXiv:math/0301146[math.AG]]}


\bibitem[Qi]{Qi} Hua Qiang,
\newblock {\em Bott-Chern characteristic classes for coherent sheaves}, \\ 
\href{https://arxiv.org/abs/1611.04238}
{[arXiv:1611.04238[math.DG]]}

\bibitem[Q]{Q} D.\,Quillen,
{\em Superconnections and the Chern character}\/,
Topology {\bf 24} (1985)
89-95

\bibitem[R]{Ros} A.\ Rosly,
{\em Superconnections and Chern classes of coherent sheaves,} \\
Talk at: Categorical and Analytic Invariants in Algebraic Geometry, II, \\
International Conference, Japan, Kashiwa, 16-20 November 2015


\bibitem[Sab]{Sabbah}  C.\,Sabbah, {\em Introduction to the theory of D-modules},  
(Lecture notes Nakai 2011) \\
\href{https://perso.pages.math.cnrs.fr/users/claude.sabbah/livres/sabbah_nankai110705.pdf}
{https://perso.pages.math.cnrs.fr/users/claude.sabbah/livres/sabbah\_nankai110705.pdf}

\bibitem[S]{Schuster} H.-W.\,Schuster, 
{\em Locally free resolutions of coherent sheaves on surfaces}, \\ 
J.\,Reine\,Angew.\,Math. {\bf 337} (1982) 159-165



\bibitem[TV]{TV2} B.\,To\"en, M.\,Vaqui\'e, \\ 
{\em Alg\' ebrisation des vari\' et\' es analytiques complexes et
cat\' egories {\mbox d\' eriv\' ees}}, \\ 
Math.\,Ann. {\bf 342} (2008)
789-831
\href{https://arxiv.org/abs/math/0703555}
{[arXiv:math/0703555[math.AG]]}

\bibitem[V]{Verb1} M.\,Verbitsky,
{\em Coherent sheaves on general $K3$ surfaces and tori}\/, \\
Pure\,Appl.\,Math.\,Q. {\bf 4} (2008) 
651-714
\href{https://arxiv.org/abs/math/0205210}
{[arXiv:math/0205210[math.AG]]}



\bibitem[Vo]{Voisin} C.\, Voisin,
{\em A counterexample to the Hodge conjecture extended to Kahler varieties}\/,
International\,Math.\,Research\,Notes {\bf 20} (2002) 1057-1075 \\
\href{https://arxiv.org/abs/math/0112247}
{[arXiv:math/0112247[math.AG]]}

\end{smallbibl}
{\bf Alexey Bondal}\\
Steklov Mathematical Institute of Russian Academy of Sciences, Moscow, Russia, and \\
	Center of Pure Mathematics, Moscow Institute of Physics and Technology, Russia, and\\
	Kavli Institute for the Physics and Mathematics of the Universe (WPI), The University of Tokyo, Kashiwa, Chiba 277-8583, Japan
\\ {\it Email address:} bondal@mi-ras.ru

\noindent
{\bf Alexei Rosly}\\ 
Skoltech, IITP RAS, and HSE University, Moscow, Russia
\end{document}